\documentclass[10pt]{amsart}
\usepackage{graphicx}
\usepackage{amssymb}
\usepackage{amsmath,amsthm}
\usepackage{hyperref}
\usepackage{amsthm}
\usepackage[english]{babel}
\usepackage{mathrsfs}

\newtheorem{thm}{Theorem}[section]
\newtheorem*{thm*}{Theorem}
\newtheorem{cor}[thm]{Corollary}
\newtheorem*{cor*}{Corollary}
\newtheorem{lem}[thm]{Lemma}
\newtheorem{prop}[thm]{Proposition}
\theoremstyle{definition}
\newtheorem{defn}[thm]{Definition}
\newtheorem{rem}[thm]{Remark}
\newtheorem*{note}{Note}

\newtheorem{examples}[thm]{Examples}

\newtheorem{notation}[thm]{Notation}

\def\ra{\rightarrow}

\title[$\rm II_1$ factors of negatively curved groups, II]{On the structural theory of $\rm II_1$ factors of \\ negatively curved groups, II. Actions by product groups}
\author[I. Chifan]{Ionut Chifan}
\address{Ionut Chifan, University of Iowa,14 McLean Hall, Iowa City, IA, USA 52242 and IMAR, Bucharest, Romania}
\email{ionut-chifan@uiowa.edu}
\author[T. Sinclair]{Thomas Sinclair}
\address{Thomas Sinclair, Department of Mathematics, UCLA, Los Angeles, CA, USA 90095--1555}
\email{thomas.sinclair@math.ucla.edu}
\author[B. Udrea]{Bogdan Udrea}
\address{Bogdan Udrea, University of Iowa,  14 McLean Hall, Iowa City, IA, USA 52242 and IMAR, Bucharest, Romania}
\email{bogdan-udrea@uiowa.edu}
\date{\today}

\dedicatory{}
\newcommand{\ca}{\curvearrowright}
\newcommand{\de}{\delta}
\newcommand{\up}{\upsilon}

\newcommand{\G}{\Gamma}
\newcommand{\g}{\gamma}
\newcommand{\e}{\varepsilon}

\newcommand{\s}{\sigma}
\newcommand{\bb}{\mathbb}

\newcommand{\La}{\Lambda}
\newcommand{\la}{\lambda}

\newcommand{\Aut}{\operatorname{Aut}}

\newcommand{\id}{\operatorname{id}}

\DeclareMathOperator*{\Lim}{Lim}

\newcommand{\Prob}{\operatorname{Prob}}
\newcommand{\SL}{\operatorname{SL}}
\newcommand{\Sp}{\operatorname{Sp}}

\DeclareMathOperator*{\res}{\upharpoonright}

\newcommand{\ip}[2]{\langle #1, #2 \rangle}
\providecommand{\abs}[1]{\lvert #1 \rvert}
\providecommand{\nor}[1]{\lVert #1 \rVert}

\begin{document}
\begin{abstract} This paper contains a series of structural results for von Neumann algebras arising from measure preserving actions by product groups on probability spaces. Expanding upon the methods used earlier by the first two authors, we obtain new examples of strongly solid factors as well as von Neumann algebras with unique or no Cartan subalgebra. For instance, we show that every II$_1$ factor associated with a weakly amenable group in the class $\mathcal S$ of Ozawa is strongly solid. There is also the following product version of this result: any maximal abelian $\star$-subalgebra of any II$_1$ factor associated with a finite product of  weakly amenable groups in the class $\mathcal S$ of Ozawa has an amenable normalizing algebra. Finally, pairing some of these results with a cocycle superrigidity result of Ioana, it follows that compact actions by finite products of lattices in $Sp(n, 1)$, $n \geq2$, are virtually $W^*$-superrigid. \end{abstract}

\maketitle\tableofcontents

\section*{Introduction}

 An important motivation in the study of $\rm II_1$ factors---in fact, one of von Neumann's original motivations in inventing the subject---is that they provide an analytical and algebraic framework for the representation theory of groups and ergodic theory. The usefulness of this observation lies in the fact that classification questions in ergodic theory or representation theory can often be reformulated as questions on the algebraic structure of certain von Neumann algebras, and that such questions may be approached with strategies and techniques beyond those which are available in the standard ergodic or representation-theoretic toolkits. A notable example of the translation of problems from ergodic theory to the theory of von Neumann algebras is the fundamental result of Singer \cite{Sing55} which states that the orbit equivalence class of a free, ergodic, p.m.p.\ action of a countable discrete group is in one-to-one correspondence with the group of automorphisms of a canonical associated $\rm II_1$ factor which preserves a canonical subalgebra. Thus, the problem of characterizing all group actions orbit equivalent to a given one is reduced to the calculation of the symmetry group of some algebraic object.

With such applications in mind, Popa developed in the first half of the last decade a powerful theory for the classification of algebraic structure in $\rm II_1$ factors which he termed deformation/rigidity \cite{PoBe, PoI, PoICM}. Popa's techniques rapidly led to the settling of several long-standing problems in the theory of $\rm II_1$ factors \cite{PoBe} as well as far-reaching classification results in the orbit equivalence theory of ergodic actions, notably, Popa's cocycle superrigidity theorems \cite{PoCSR,PoSG}.  Following Popa's seminal work, the classification of $\rm II_1$ factors has witnessed a rebirth. To list some of the major accomplishments which have occurred in the last several years: the cocycle superrigidity theorems of Popa \cite{PoCSR,PoSG} and Ioana \cite{IoaCSR}; work on the classification of Cartan subalgebras by Ozawa and Popa \cite{OPCartanI, OPCartanII}; the discovery $\rm W^*$-superrigid groups and actions with substantial contributions by Ioana, Peterson, Popa, and Vaes, among others, \cite{PeWS,PoVaWS,IoaWS,IPV,CP,Vae}; and the study of various structural properties for von Neumann algebras such as strong solidity initiated by Ozawa and Popa \cite{OPCartanI,OPCartanII} and continued by others \cite{Hou,HouSh,Sin,CS}. 

This paper is the continuation of an article \cite{CS} by the first two authors. The broad theme of that article was the application of geometric techniques in the context of Popa's deformation/rigidity theory to obtain structural results for $\rm II_1$ factors associated to Gromov hyperbolic groups and their actions on measure spaces. This was accomplished in part through the reinterpretation of Ozawa's $\rm C^*$-algebraic structural theory of group factors \cite{OzSolid, OzKurosh} in terms of Peterson's cohomological approach \cite{PetL2} to Popa's deformation/rigidity theory. However, partly for reasons of clarity, there are aspects of Ozawa's theory which were not touched upon in the previous paper---specifically, the use of ``small'' families of subgroups to unify various structural theorems \cite{BrOz, Sako}. The aim of this paper is to incorporate these techniques into the deformation/rigidity approach in \cite{CS}. 
The main applications which will be addressed are to p.m.p.\ actions of countable discrete groups $\G$ which fall into two basic cases: (1) $\G$ is generated by a pair of subgroups $(G_1,G_2)$ which are rather ``free'' with respect to each other (precisely, $\G$ is relatively hyperbolic to $\{G_1,G_2\}$), or (2) $\G$ is generated by a pair of ``negatively curved'' groups $\{G_1,G_2\}$ with a high degree of commutation. Aside from this we will also be able the sharply generalize most of the results in the previous paper to cover a more general class of groups.

\subsection*{Statement of results} The main result of this paper will be the following theorem which improves Theorem B/Theorem 4.1 of \cite{CS} in two ways. First, we are able to extend the theorem to the more general class of exact groups which admit proper \emph{arrays} into weakly $\ell^2$ representations (i.e., bi-exact groups) rather than just proper quasi-cocycles. Secondly, we are able to deal with groups which are ``negatively curved'' with respect to a collection of ``small'' subgroups, which includes, primarily, the widely studied class of relatively hyperbolic groups \cite{Bow, Farb}. The result and its proof are inspired by Ozawa's general semi-solidity theorem (Theorem 15.1.5 in \cite{BrOz}) and Ozawa and Popa's Theorem B in \cite{OPCartanII}, viewed through the framework developed by the first two authors in \cite{CS}. 

\begin{thm}[Theorem 6.1]
Let $\,\G$ be an exact group, let $\pi : \G \ra \mathcal U(\mathcal H_\pi)$ be a weakly-$\ell^2$ representation and assume that one of the following holds: 
\begin{enumerate}
\item  $\G$ admits a proper array into $\mathcal H_\pi$ (i.e., $\mathcal{RA}(\G,\{e\},\mathcal H_\pi)\neq\emptyset$: see \S\ref{sec:relarrays}); or,
 \item there exists  $\,\mathcal G$, a family of subgroups of $\,\G$ such that $\G$ admits a quasi-cocycle which is metrically proper relative to the length metric coming from the generating set $S = \bigcup_{\Sigma\in \mathcal G} \Sigma$ (i.e., $\,\mathcal{RQ}(\G,\mathcal G,\mathcal H_\pi)\neq\emptyset$: see \S\ref{sec:relarrays}). \end{enumerate}Also let $\,\G \ca X$ be a free, ergodic p.m.p.\ action, denote by $M = L^\infty(X)\rtimes \G$ the corresponding crossed-product von Neumann algebra, and let $\,P \subseteq M$ be any weakly compact embedding with $P$ diffuse. Then the following holds:
 \begin{itemize} \item if $\,\G$ satisfies condition (1) above then either the normalizing algebra $\mathcal N_M (P)''$ is amenable or $P\preceq_M L^\infty(X)$.
\item if $\,\G$ satisfies condition (2) above then either the normalizing algebra $\mathcal N_M (P)''$ is amenable or there exists a group $\Sigma\in {\mathcal G}$ such that $P\preceq_M L^\infty(X)\rtimes \Sigma$.\end{itemize}\end{thm}

\noindent As a consequence any free ergodic weakly compact action\cite{OPCartanI} of any weakly amenable group $\G$ in the class $\mathcal S$ of Ozawa \cite{OzKurosh} gives rise to a von Neumann algebra with unique Cartan subalgebra. Moreover for all these groups as well as  all $\G$ that are hyperbolic relative to a collection of subgroups which are, in some sense, peripheral (cf.\ \S1.2 and \S4 in \cite{Farb}), then $L\G$ is strongly solid, as the following corollary demonstrates. For instance this will be the case when $\G$ is any group in the measure equivalence class of an arbitrary limit group in the sense of Sela. These groups should be considered as generalizations of non-uniform lattices in rank one Lie groups, which may admit finitely many cusp subgroups.

\begin{cor}[Corollary 6.8] Let $\G$ be a weakly amenable group and let 
$\pi:\G\ra \mathcal U(\mathcal H_\pi)$ be an weakly-$\ell^2$ representation such that  one of the following holds: either $\, \mathcal {RA}(\G,\{e\}, \mathcal H_\pi)\neq\emptyset$, or  there exists  $\,\mathcal G$, a family of amenable, malnormal subgroups of $\,\G$ such that $\, \mathcal {RQ}(\G,\mathcal G, \mathcal H_\pi)\neq\emptyset$. If $\,\La$ is any $ME$-subgroup of $\,\G$ then $L\La$ is strongly solid i.e., given any diffuse amenable subalgebra $\, A\subseteq L\La$ its normalizing algebra $\mathcal N_{L\La}(A)''$ is still amenable. In particular, every amenable subgroup of $\La$ has amenable normalizer.  \end{cor}

Our techniques also allow us to obtain structural results for normalizers in direct products of negatively curved groups. Such groups are interesting in that they provide highly tractable examples of groups which exhibit higher-rank (rigid) phenomena (cf.\ \cite{CP, PoSG, MScocycle, MSoe}). On the other hand, the next result will show that the structure of their group factors may be reduced to the study of their rank one components (this ``rank one'' decomposition is algebraically unique by \cite{OPPrime}; see also Theorem C in \cite{CS}). The result is optimal, though more intricate to state than the previous, since one needs to account for the presence of commutation between the factors.

\begin{thm}[Theorem 6.5] \label{control-weak-comp-embed}For every $i=1,2$ let $\G_i$ be an exact group such that $\mathcal{RA}(\G_i,\{e\},\ell^2(\G_i))\not= 0$. Let $(\G_1\times \G_2) \curvearrowright X$ be a free, ergodic, p.m.p.\ action and denote by $M=L^\infty (X)\rtimes (\G_1\times \G_2)$ the corresponding crossed-product von Neumann algebra. If $P \subseteq M$ is any weakly compact embedding with $P$ diffuse, then one can find projections $p_0$, $p_1$, $ p_2$, $p_3\in \mathcal Z(\mathcal N_M(P)'\cap M)$ with $p_0+p_1+p_2+p_3=1$ such that:
\begin{enumerate}
\item $\mathcal N_M(P)''p_0$ is amenable;
\item $Pp_1\preceq_M L^\infty(X)\rtimes \G_1$;
\item $Pp_2\preceq_M L^\infty(X)\rtimes \G_2$;
\item $Pp_3\preceq_M L^\infty(X)$.
\end{enumerate}
\end{thm}

The next corollary, for the special case of tensor products of free group factors, is an unpublished result of Ioana and the first author \cite{CIfree}. It would be interesting to know whether the result also holds true for generic higher rank lattices, e.g. $\SL(3,\bb Z)$.

\begin{cor}\label{masanormalizer} Let $\, \G_1$, $\G_2$ be i.c.c.\ hyperbolic groups and denote by $M  =L\G_1\bar\otimes L\G_2$. If $A\subseteq M$ is an amenable subalgebra such that  $A'\cap M$ is amenable (e.g. when $A$ is either a m.a.s.a.\ or an irreducible, amenable subfactor of $M$) then its normalizing algebra  $\mathcal N_M(A)''$ is amenable.  
\end{cor}

The following corollary is complementary to Corollary 6.2 in \cite{CP}, which holds for product actions of rigid groups which are sufficiently mixing. Interestingly, for actions between these two extremes, the result is known to fail (Example 2.22 in \cite{MSoe}).

\begin{cor}\label{$W^*$-strongrigidity} If  $\, \G_1$, $\G_2$ are hyperbolic groups with property (T) (e.g. $\G_i$ lattices in $\,\Sp(n,1)$ $n\geq 2$),  then any free, ergodic, profinite (or more generally compact) action $(\G_1\times \G_2)\curvearrowright X $ is virtually $W^*$-superrigid. \end{cor}

\section{Popa's Intertwining Techniques} 
 We will briefly review the concept of intertwining two subalgebras inside a von Neumann algebra, along with the main technical tools developed by Popa in \cite{PoBe,PoI}. Given $N$ a finite von Neumann algebra, let $P\subset
fNf$, $Q\subset N$ be diffuse subalgebras for some projection $f\in N$. We say that \emph{a corner of $P$ can be intertwined into $Q$ inside $N$} if there exist two non-zero projections $p\in P$, $q\in Q$, a non-zero partial isometry $v\in pNq$, and a $\star$-homomorphism $\psi:pPp\ra qQq$ such that $v\psi(x)=xv$ for all $x\in pPp$. Throughout this paper we denote by $P\preceq_{N}Q$ whenever this property holds, and by $P\npreceq_{N}Q$ its negation. The partial isometry $v$ is called an intertwiner between $P$ and $Q$.  

Popa established an efficient criterion for the existence of such intertwiners (Theorems 2.1-2.3 in \cite{PoI}). Particularly useful in concrete applications is the following \emph{analytic} description of absence of intertwiners. 

\begin{thm}[Corollary 2.3 in  \cite{PoI}]\label{intertwining-techniques} Let $N$ be a von Neumann algebra and let  $P\subset
fNf$, $Q\subset N$ be diffuse subalgebras for some projection $f\in N$. Then the following are equivalent:
\begin{enumerate}
\item $P\npreceq_M Q$;
\item For every finite set $\mathcal F\subset Nf$ and every $\epsilon>0$ there exists a unitary $v\in \mathcal U(P)$ such that  
\begin{equation*}\sum_{x,y\in\mathcal F} \|E_Q(xvy^*)\|^2_2\leq \epsilon. \end{equation*}  \end{enumerate}
\end{thm}

Notice that in the intertwining concept presented above we \emph{a priori} have no control over the image $\psi(pPp)$ inside $qQq$. When trying to get unitary conjugacy this often becomes a significant issue and additional analysis regarding the position of $\psi(pPp)$ inside $qQq$ is required. Sometimes the $\star$-homomorphism $\psi$ can be suitably modified to automatically  preserve certain properties from the inclusion $P\subset N$ to the inclusion $\psi(pPp)\subseteq qQq$. For instance, Ioana showed in Lemma 1.5 of \cite{ioana2011} that if $P\subset N$ is a m.a.s.a.\ then  $\psi$ can be chosen so that $\psi(pPp)\subseteq qQq$ is again a m.a.s.a.\ Applying his argument one can show that $\psi$ can be chosen to also preserve the irreducibility of the inclusion $P\subset N$. The precise technical result which will be of essential use to derive some of our main applications is the following:   

\begin{prop}\label{irredtransfer}Let $N$ be a von Neumann algebra together with subalgebras $P,Q \subseteq N$ such that $P'\cap N=\mathbb C 1$. If we assume that $P\preceq_N Q$ then one can find projections $p \in P$, $q \in Q$, a $\star$-homomorphism $\phi : pPp \rightarrow qQq$ and a non-zero partial isometry $v \in qNp$ such that $\phi(x)v = vx$, for all $x \in pPp$, and $\phi(pPp)'\cap qQq=\mathbb C q$.\end{prop}

\noindent The proof of this result follows the same strategy as the proof of Lemma 1.5 of \cite{ioana2011}, so it will be omitted.  

We end this section by recalling two important intertwining results from the work of Popa \cite{PoBe,PoI}. These results play a very important role in deriving some of our main applications. The first result describes an inclusion of von Neumann algebras  where we have complete control over general intertwiners of subalgebras. To properly introduce the statement we need a definition. Given an inclusion of countable groups $\Sigma<\G$ we say that $\Sigma$ is \emph{malnormal} in $\G$ if and only if for every $\g\in \G\setminus \Sigma$ we have $\g\Sigma\g^{-1}\cap \Sigma$ is finite. 

\begin{prop}[Theorem 3.1 in \cite{PoI}]\label{controlmalnormal}Let $\Sigma<\G$ be a malnormal group, let  $\G\ca A$ be a trace preserving action and denote by $M=A\rtimes \G$ the corresponding crossed product von Neumann algebra. Also let $p\in A \rtimes \Sigma$ be a projection and suppose that $P\subseteq p(A\rtimes \Sigma) p$ is a diffuse subalgebra such that $P\npreceq_{A\rtimes \Sigma} A$. If there exist elements $x,x_1,x_2,\ldots,x_n \in M $  such that $Px\subseteq \sum_i x_i P$ then $x\in A\rtimes \Sigma$. 
\end{prop}

The second result which will be needed in the sequel is Popa's unitary conjugacy criterion for Cartan subalgebras. 
\begin{thm}[Appendix 1 in \cite{PoBe}]\label{intertwining-conj}
Let $N$ be a II$_1$ factor and $A,B\subset N$ two semiregular m.a.s.a.\ (i.e., their normalizing algebras $\mathcal N_N(A)''$ and $\mathcal N_N(B)''$ are subfactors of $N$). If $B_0 \subset B$ is a von Neumann subalgebra such that $B_0' \cap N =B$, and $B_0\preceq_NA$, then there exists a unitary $u\in N$ such that $uAu^*=B$.
\end{thm}


\section{Relative Arrays and Relative Quasi-cocycles}\label{sec:relarrays}

In this section we consider relative versions of the notions of arrays \cite{CS} and quasi-cocycles \cite{Mon,Min,MMS} for groups. This will allow us  to generalize, from the viewpoint of deformation/rigidity theory, the structural results obtained in \cite{CS}.  After introducing the definitions, we summarize a few useful properties, relating these with other concepts extant in the literature. In the last part of the section we will present several examples, some of them arising naturally from geometric group theory.

\subsection{Relative arrays} Assume that $\G$ is a countable, discrete group together with  $\mathcal G=\{\Sigma_i\,:\,i\in I\}$, a family of subgroups  of $\G$ and $\pi : \G\ra \mathcal U(\mathcal H)$, a unitary representation. 

\begin{defn}We say that a group $\G$ admits \emph{a proper array relative to $\mathcal G$ into $\mathcal H$} if there exists a map $r: \G\ra \mathcal H$ which satisfies the following conditions:
\begin{enumerate}
\item $\pi_\g(r(\g^{-1}))= \pm r(\g)$ for all $\g\in\G$, i.e., (anti-)symmetry;
\item for every $\g\in \G$ we have \begin{equation*}\sup_{\de\in \G}\|r(\g \de)-\pi_\g(r(\de)) \|=C(\g)<\infty;\end{equation*}   
\item the map $ \g\ra \|r(\g)\|$ is proper with respect to $\mathcal G$, i.e. for every $C>0$ there exist finite subsets $\mathcal F\subset\mathcal G$ and $H,K\subset \G$  such that  

\begin{equation*}\{\g\in \G \,:\,\|r(\g)\|\leq C \}\subseteq \cup_{\Sigma \in\mathcal F} H\Sigma K.\end{equation*}

\end{enumerate}
\end{defn}

\begin{notation} Given a map  $\phi:\G\ra \mathbb R$ and $\ell\in \bb R$, we say that \[\lim_{\g\ra \infty/\mathcal G}\phi(\g)=\ell\] if for every $\epsilon>0$ there exist  finite sets $H,K\subset \G, \mathcal F\subset \mathcal G$ such that $|\phi(\g)-\ell|<\epsilon$ for all $\g\not\in H\mathcal F K$. 
\end{notation}

From now on the set of all such relative arrays will be denoted by $ \mathcal{RA}(\G,\mathcal G,\mathcal H_\pi)$. Notice that when $\mathcal G$ consists of the trivial subgroup only, one recovers the notion of proper, (anti-)symmetric arrays as defined in \cite{CS}. For further discussion on arrays the reader may consult section 1 in \cite{CS}.

When considering exact groups, the above notion of relative array into the left regular representation is closely related with the notion of bi-exactness introduced by Ozawa (Definition 15.1.2 in \cite{BrOz}). We are indebted to Narutaka Ozawa for kindly demonstrating to us the direct implication in the following result.
\begin{prop}\label{RA=OzS} Let $\G$ be an exact group together with $\mathcal G$ a family of subgroups. Then $\mathcal{RA}(\G,\mathcal G,\ell^2(\G))\neq\emptyset$ if and only if $\G$ is bi-exact with respect to $\mathcal G$. 
\end{prop}

\begin{rem}A recent result of Popa and Vaes \cite{PoVa12} establishes the same result under the weaker assumption that $ \mathcal{RA}(\G,\mathcal G,\mathcal H_\pi)\not= \emptyset$ for some weakly-$\ell^2$ representation $\pi$.\end{rem}

\begin{proof} The reverse implication can be shown using the same method as in \cite{CS} and therefore we only prove the direct implication. So let $r:\G\ra \ell^2(\G)$ an array relative to the family $\mathcal G$ and denote by $\pi:\G\ra \mathcal U (\ell^2(\G))$ the left regular representation. Let $\Prob(\G)$ be the set of positive Borel probability measures on $\G$.  For any $f\in\ell^\infty(\G)$ we let $\phi(f)$ the natural ``diagonal'' operator acting by pointwise multiplication.

Then we define the map $\mu:\G\ra \Prob(\G)$ by letting \begin{equation*}\langle\mu(\g),f\rangle=\frac{1}{\|r(\g)\|^2}\langle\phi(f)r(\g),r(\g)\rangle,\end{equation*} 
for all  $\g\in \G$ and $f\in\ell^\infty(\G)$. 
Also if we fix $s,t\in \G$ then denoting by $C_s=\sup_{\g\in \G}\|r(s\g)-\pi_s(r(\g))\| $ and using the triangle inequality together with the (anti-)symmetry of the array $r$ we have that
\begin{equation}\label{1010}\begin{split} \|r(s\g t)-\pi_s(r\g)\| &\leq \|r(s\g t)-r(s\g)\|+\|r(s\g)-\pi_s(r(\g))\| \\
&= \|\pi_{s\g t}(r(t^{-1}\g^{-1}s^{-1})) -\pi_{s\g}(r(\g^{-1}s^{-1}))\|+C_s\\
&= \|\pi_{ t}(r(t^{-1}\g^{-1}s^{-1}))-r(\g^{-1}s^{-1})\|+C_s\\
&\leq C_t+C_s,\end{split}
\end{equation} 
for all $\g\in\G$. 

In the remaining part we will use this estimate to show that for all $s,t\in\G$ we have   
\begin{equation}\label{1011}\lim_{\g\ra \infty/\mathcal G}\|\mu(s\g t)-s\cdot\mu(\g)\|=0,\end{equation} which in turn will give the desired conclusion.  

To see this we  fix $s,t, \g \in \G$ and $f\in\ell^\infty(\G)$. Then applying the triangle inequality in combination with (\ref{1010}) and the Cauchy-Schwarz inequality we have 

\begin{eqnarray*}&&\left |\langle \mu(s\g t),f\rangle -\langle s\cdot \mu(\g),f\rangle \right|\\ &\leq &\frac{1}{\|r(s\g t)\|^2}\left |\langle\phi(f)r(s\g t),r(s\g t)-\pi_s(r(\g))\rangle\right |+\\ && +\left |\left (\frac{1}{\|r(s\g t)\|^2}-\frac{1}{\|r(\g)\|^2}\right )\langle\phi(f)r(s\g t),\pi_s(r(\g))\rangle \right |+\\ &&+\frac{1}{\|r(\g)\|^2}\left | \langle\phi(f)r(s\g t)-\pi_s(r(\g)),r(\g )\rangle \right | \\ 
&\leq & \frac{C_s+C_t}{\|r(s\g t)\|^2}\|\phi(f)r(s\g t)\|+\left |\frac{1}{\|r(s\g t)\|^2}-\frac{1}{\|r(\g)\|^2}\right | \|\phi(f)r(s\g t)\| \|r(\g)\| +\\
&&+\frac{1}{\|r(\g)\|}\|\phi(f)r(s\g t)-\pi_s(r(\g))\| \\ 
&\leq &  2(C_s+C_t)\|f\|\left (\frac{1}{\|r(s\g t)\|}+\frac{1}{\|r(\g )\|}\right ).  \end{eqnarray*}

Since $r$ is assumed to be proper with respect to the set $\mathcal G$ then $ \lim_{\g\ra\infty/\mathcal G}\|r(s\g t)\|=\infty$, $ \lim_{\g\ra\infty/\mathcal G}\|r(\g )\|=\infty$ and thus taking the limit in the previous inequality we get (\ref{1011}). \end{proof}

\subsection{Relative quasi-cocycles}In the same spirit, if $\G$ is a group together with a family of subgroups $\mathcal G = \{\Sigma_i\,:\, i\in I\}$ and a unitary representation $\pi : \G \ra \mathcal U(\mathcal H)$, we say that pair $(\G,\mathcal G)$ admits \emph{a relative quasi-cocycle into $\mathcal H$} if there exists a map $r: \G\ra \mathcal H$ satisfying the following condition:
\begin{enumerate}
\label{10000} \item there exists a constant $C>0$ such that \begin{equation*}\sup_{\g,\de\in \G}\|r(\g \de)-\pi_\g(r(\de))-r(\g) \|\leq C.\end{equation*}  

\item the map $ \g\ra \|r(\g)\|$ is proper relative to $\mathcal G$.
\end{enumerate}

From now on, the set of all such relative quasi-cocycles we will denoted by $\mathcal {RQ}(\G,\mathcal G,\mathcal H_\pi)$. Using the terminology from \cite{Tho}, it is clear that $\mathcal {RQ}(\G,\mathcal G,\mathcal H_\pi)$ is a subset of $\mathcal QH^1(\G,\mathcal H_\pi)$ which is stable under scalar multiplication and translation by uniformly bounded maps, without being in general a vector subspace. It is also straightforward that every relative quasi-cocycle is a relative array, i.e., we always have $\mathcal {RQ}(\G,\mathcal G,\mathcal H_\pi)\subseteq\mathcal {RA}(\G,\mathcal G,\mathcal H_\pi)$. The next proposition summarizes a few basic properties which follow directly from definitions. 

\begin{prop}\label{basicproprq}For each $n\in\mathbb N$ let $\mathcal G_n$ be a family of subgroups of $\G$ together with  $\pi_n:\G\ra \mathcal U(\mathcal H_n)$ a unitary representation. Then we have the following:
\begin{enumerate}
\item If $\mathcal G_1\subset \mathcal G_2$ then  $\mathcal {RA}(\G,\mathcal G_1,\mathcal H_{\pi_1})\subseteq \mathcal {RA}(\G,\mathcal G_2,\mathcal H_{\pi_1})$;
\item If  $r\in\mathcal {RA}(\G,\mathcal G_1,\mathcal H_1)$ and $c:\G\ra \mathcal H_1$ is a uniformly bounded map then $r+c\in\mathcal {RA}(\G,\mathcal G,\mathcal H)$; 
\item If $\mathcal G_n=\mathcal G_1$ and $\pi_n=\pi_1$ for all $n$ and there exists a sequence $r_n\in\mathcal {RA}(\G,\mathcal G_1,\mathcal H_{\pi_1})$ with uniformly bounded defects such that  $r_n$ converges to $r$ uniformly then $r\in\mathcal {RA}(\G,\mathcal G_1,\mathcal H_{\pi_1})$; 
\item Denote by $\wedge_n \mathcal G_n=\{\Sigma_1\cap \bigcap_{j\neq 1} s_j\Sigma_js_j^{-1}\, : \, \Sigma_1\in \mathcal G_1, \Sigma_j\in \mathcal G_j, s_j\in \G \}$. If for every $n\in \mathbb N$ there exists $c_n>0$ and $r_n\in \mathcal {RA}(\G,\mathcal G,\mathcal H_{\pi_n})$ satisfying $\sum_nc^2_n \|r_n(\g)\|^2<\infty$ for all $\g\in \G
$, then $$\mathcal {RA}(\G,\wedge_n\mathcal G_n,\oplus\mathcal H_{\pi_n})\neq \emptyset.$$  
\end{enumerate}
 \end{prop}

Cocycles, quasi-cocycles, and arrays combine both geometric and representation-theoretical data in a way that can be used to efficiently extract information about a group's internal structure. For instance, by the same proof as in Proposition 1.5.3 of \cite{CS} we can locate centralizers of certain subgroups and, in some cases, even normalizers. This property, generically called the ``spectral gap rigidity principle'', is the main intuition for the von Neumann algebraic structural results obtained in the subsequent sections.

\begin{prop}Let $\G$ be a countable group, $\mathcal G$ be a family of subgroups, and $\pi:\G\ra \mathcal U(\mathcal H_\pi)$ a representation such that $\mathcal{RA}(\G,\mathcal G,\mathcal H_\pi)\neq\emptyset$. If $\La <\G$ is a subgroup such that $1 \nprec \pi\res_\La$ then there exists $h\in \G$ and $\Sigma \in\mathcal G$ such that its centralizer $C_\G(\La)$ satisfies $[C_\G(\La)\,: \,h\Sigma h^{-1}\cap C_\G(\La)]<\infty$.  \end{prop}

\begin{proof}Let $q:\G\ra \mathcal H_\pi$ be an array. Since  $1 \nprec \pi\res_\La$ there exists a finite, symmetric subset $S\subset \La$ and $K'>0$ such that 
\begin{equation}\|\xi\|\leq K' \sum_{s\in S}\|\pi_s(\xi)-\xi \|, \text{ for all } \xi \in\mathcal H_\pi.\end{equation}

Since $q$ is an array there exists $K''>0$ such that $\|q(s\g)-\pi_s(q(\g))\|\leq K''$ for all $s\in S$ and $\g\in \G$. Set $K=\max\{K',K''\}$. Then, for every $\g\in C_\G(\La)$ we have that 

\begin{equation*}\begin{split}\|q(\g)\|&\leq K \sum_{s\in S}\|\pi_s(q(\g))-q(\g) \| \\ 
&\leq K \sum_{s\in S}\|q(s\g))-q(\g) \|+ K^2|S|\\
&\leq K \sum_{s\in S}\|q(\g s))-q(\g) \|+ K^2|S|\\
&\leq K \sum_{s\in S}\|\pi_{\g s}q(s^{-1}\g^{-1}))-\pi_\g (q(\g^{-1})) \|+ K^2|S|\\
&= K \sum_{s\in S}\|\pi_{ s}q(s^{-1}\g^{-1}))-q(\g^{-1}) \|+ K^2|S|\\
&\leq 2K^2|S| \end{split}\end{equation*}
  
 This shows that $q$ is bounded on $C_\G(\La)$ and since $q$ is proper with respect to the family $\mathcal G$. It follows that $C_\G(\La)$ is small with respect to the family $\mathcal G$ which means that there exists a finite collection of groups $\Sigma_i \in\mathcal G $ and a finite set of elements $h_i,k_i\in \G$ such that $C_\G(\La)\subseteq \bigcup_i h_i\Sigma_ik_i$. Therefore, if we denote by $\Omega_i=h_i\Sigma_ih_i^{-1}$, there exists a finite set of elements $\ell_i\in \G$ such that  $C_\G(\La)\subseteq \bigcup_i \Omega_i\ell_i$. In particular this implies that  $C_\G(\La)= \bigcup_i \left (\Omega_i\ell_i\cap C_\G(\La)\right )$. After dropping all the empty intersections we can assume that $\Omega_i\ell_i\cap C_\G(\La)\neq \emptyset$, for all $i$. Hence there exists $s_i\in \Omega_i$ such that $r_i=s_i\ell_i\in C_\G(\La)$ and we obviously have that $$C_\G(\La)= \bigcup_i \left (\Omega_ir_i\cap C_\G(\La)\right )=\bigcup_i \left (\Omega_i\cap C_\G(\La)\right )r_i.$$
 Finally, by Lemma 4.1 in \cite{BHNe}, the previous relation implies that $\Omega_i\cap C_\G(\La)$ have finite index in $C_\G(\La)$ and we are done. \end{proof}

\noindent Here are two concrete situations when this happens: $\La$ has property (T) and the restriction $\pi\res_\La$ has no invariant vectors; $\La$ is not co-amenable with respect to a subgroup $\Sigma<\G$ and $\pi$ is the left semi-regular representation $\ell^2(\G/\Sigma)$.   

Moreover, if $\G$ is weakly amenable (cf.\ \S 3), $\La$ is amenable, and the normalizing group satisfies $1\nprec \pi\res_{N_\G(\La)}$ then $\La$ is small with respect to $\mathcal G$. This is an easy consequence of Corollary 3.2 in the sequel but in the in this form can be shown by direct arguments similar to the above proof.

\begin{examples}There are many examples of groups that admit relative quasi-cocycles (arrays) into various representations. First we analyze a few examples arising from canonical group constructions:    
\vskip 0.05in
 \noindent {\bf A. Exact sequences.}  Let $L,K, \G$ be groups such that $0\rightarrow L\rightarrow K\rightarrow \G \ra 0$ is a short exact sequence.  If $\mathcal{RA}(\G, \{ e \},\ell^2(\G))\neq \emptyset$  then we have  $\mathcal{RA}(K, \{ L \},\ell^2(K/L))\neq \emptyset$.
 
 \vskip 0.05in
 
 \noindent {\bf B.  Product groups.} Let $\G_1,\G_2,\ldots,\G_n$ be a collection of groups, and denote by $\G=\G_1\times \G_2\times\cdots \times \G_n$. For every $1\leq i\leq n$ we denote by $\hat \G_i$ the subgroup of the direct product $\G$ which consists of all elements whose $i^{th}$ coordinate is trivial. Assume that $\mathcal G_i$ is family of subgroups of $\G_i$ and denote by $\mathcal G= \bigcup_i\{\La\times \hat \G_i\, : \,\La \in \mathcal G_i\}$.  If $\mathcal {RA}(\G_i,\mathcal G_i,\mathcal H_i)\neq\emptyset$ for all $1\leq i \leq n$, then $\mathcal {RA}(\G,\mathcal G,\otimes_i\mathcal H_i)\neq \emptyset$. For the proof of this fact see Proposition 1.10 in \cite{CS}. In particular $\G_1\times \G_2$ admits an array into  a weakly-$\ell^2$ representation that is proper with respect to $\{\G_1,\G_2\}$ whenever $\G_1$ and $ \G_2$ admit proper arrays into weakly-$\ell^2$ representations.

\vskip 0.05in

\noindent {\bf C. Semidirect products.} Let $\G$ and $A$ be countable discrete groups together with $\mathcal G$ a family of subgroups of $\G$ and assume that $\rho: \G \ra \Aut(A)$ is an action by group automorphisms. Let $\pi:\G\ra \mathcal U(\mathcal H_\pi)$ be a unitary representation and define $\tilde{\pi}: A\rtimes_\rho \G\ra \mathcal U(\mathcal H_\pi)$ by letting $\tilde{\pi}_{a\g}(\xi)=\pi_\g(\xi)$ for every $a\in A$, $\g\in \G$ and $\xi\in\mathcal H_\pi$. If $c\in \mathcal {RA}(\G,\mathcal G,\mathcal H_\pi )$ then the formula $\tilde{c}(a\g)=c(\g)$ defines an array which belongs to $\mathcal {RA}(A\rtimes_\rho \G, \{A\rtimes_\rho \Sigma \, : \, \Sigma \in \mathcal G \},\mathcal H_{\tilde{\pi}})$.

We now look at semidirect products by finite groups. So let $\G$ be a countable discrete group together with a family of subgroups $\mathcal G$,  $\La$ be a finite group, and $\rho: \La \ra \Aut(\G)$ be an action by automorphisms. It is an exercise for the reader to check that for any $r\in\mathcal {RA}(\G,\mathcal G,\ell^2(\G))$, the map $$r'(\g\alpha )=\frac{1}{|\La|^\frac{1}{2}}\sum_{\de\in\La}\la_\de(r(\rho_{\de^{-1}}(\g)))$$ defines an array belonging to $\mathcal{RA}(\G\rtimes_\rho\La,\mathcal G, \ell^2(\G\rtimes_\rho \La))$, where $\g\in\G,\alpha\in\La$ and $\la$ is the left regular representation on $\ell^2(\G\rtimes_\rho \La)$.

\vskip 0.05in

\noindent {\bf D. Free products.} Let $\{\G_n\}_{1\leq i\leq n}$ be a finite collection of groups. Denote by $\G=\star_i\G_i$ their free product, and let $\pi:\G\ra \mathcal U(\mathcal H_\pi)$ a unitary representation. If for every  $1\leq i\leq n$ we have  $\mathcal{RQ}(\G_i, \{e\},\mathcal H_\pi)\neq \emptyset$, then the proof of Lemma 5.1 and Theorem 5.3 in \cite{Tho} show that $\mathcal{RQ}(\G,   \{e\},\mathcal H_\pi^{\oplus n})\neq \emptyset$. Note that the when considering quasi-cocycles proper with respect to families of subgroups, it is not clear whether the resulting quasi-cocycle is proper to any canonical family of subgroups rather than just finite length subsets over the families of subgroups we started with. However, if we assume that $\Sigma\lhd \G_i$ is a common normal subgroup, $\G= \star_\Sigma\G_i$ is the amalgamated free product over $\Sigma$, and for every $1\leq i\leq n$ we have $\mathcal{RQ}(\G_i/\Sigma, \{e\},\ell^2(\G_i/\Sigma))\neq \emptyset$ then $\mathcal{RQ}(\G,   \{\Sigma\},\ell^2(\G/\Sigma) )\neq \emptyset$. In connection to this notice that if $\Sigma$ is an amenable (non necessarily normal) subgroup then it follows from \cite{BrOz} that $\mathcal{RA}(\G, \{G_i\,:\, 1\leq i\leq n\}, \ell^2(\G)^\infty)\neq \emptyset$. 

\vskip 0.05in

\noindent {\bf E. HNN-extensions.} Denote by $\G=(H,L,\theta)$ the HNN-extension associated with a given inclusion groups $L<H$ and a monomorphism $\theta: L \ra H$. We also assume that $K\lhd H$ is a normal subgroup which contains $L$ and $\theta(L )$ and from now on we will denote by $L_1=L$, $L_{-1}=\theta(L)$. The group $\G$ may be presented as $\{H,t \,:\, \theta(\ell)=t\ell t^{-1} , \ell\in L\}$. By Britton's Lemma, every element $\g\in \G$ has a canonical reduced form $\g=\g_0t^{\e_1}\g_1t^{\e_2} \ldots\g_{n-1}t^{\e_n}\g_n$, where $\g_i\in H$, $\e_i\in\{-1,1\}$ and whenever $\e_i\neq \e_{i+1}$ we have that $\g_i\notin L_{\e_i}$, for all $1\leq i\leq n-1$. 

Assume that  $q : H/K\ra \ell^2(H/K)$ is a quasi-cocycle. By the construction in the first example there exists a quasi-cocycle $ c: H\ra \ell^2(H/K)$ which vanishes on $K$ and moreover $c\in\mathcal{RA}(H,\{K\},\ell^2(H/K))$ whenever $q$ is proper. We can define a map $r:\G\ra \ell^2(\G/K)$ in the following way: for every $\g=\g_0t^{\e_1}\g_1t^{\e_2} \ldots\g_{n-1}t^{\e_n}\g_n$ and $s=0,1$ we  let 

\begin{eqnarray*}r^s_q(\g)&=&\lambda_{\g_0t^{\e_1}\g_1t^{\e_2} \ldots\g_{n-1}t^{\e_n}} c(\g_n)+\lambda_{\g_0t^{\e_1}\g_1t^{\e_2} \ldots\g_{n-1}}d^s(t^{\e_n}) +\\&&+\lambda_{\g_0t^{\e_1}\g_1t^{\e_2} \ldots\g_{n-1}t^{\e_{n-1}}} c(\g_{n-1})+\lambda_{\g_0t^{\e_1}\g_1t^{\e_2} \g_{n-2}}d^s(t^{\e_{n-1}})+\\
&&+\ldots+\lambda_{\g_0}d^s(t^{\e_1})+c(\g_0), \end{eqnarray*}
where $d^1(t^\e)=\delta_{t^\e K}$ and $d^0(t^\e)=0$ for all $\e\in\{-1,1\}$.  
Here $\lambda$ denotes the left semi-regular representation $\ell^2(H/K)$. It is a straightforward exercise to see that this map is well defined and it satisfies the quasi-cocycle relation. Moreover, when $q=0$, the map is actually a 1-cocycle.

 Therefore, applying part (4) in Proposition \ref{basicproprq} we have that $r^1_0\oplus r^0_q$ is a quasi-cocycle into $\ell^2(\G/K)\oplus \ell^2(\G/K)$. If $q$ is assumed proper it follows that  $r^1_0\oplus r^0_q$ is proper with respect to various \emph{subsets} of $\G$, e.g. sets of words with finite length over $t$'s whose letters from $H$ are ``small'' over $K$. However, to have properness with respect to subgroups we need to impose additional assumptions on $K$.  For instance, one may assume that  $L$ and $\theta(L)$ have finite index in $K$, in which case we would have $r^1_0\oplus r^0_q\in \mathcal{RQ}(\G,\{K\},\ell^2(\G/K)\oplus \ell^2(\G/K))$.

\vskip 0.05in

\noindent  {\bf F. Inductive limits.} Let $\G_n\nearrow \G$ be an inductive limit of groups and for each $n\in \mathbb N$ let $\mathcal G_n$ be a family of subgroups of $\G_n$ such that $\mathcal G_n\subseteq \mathcal G_{n+1}$. Assume that for each $n$, there exists $r_n \in \mathcal{RA}(\G_n, \mathcal G_n,\ell^2(\G_n))$ so that:

\begin{enumerate} \item $\sup_{n,m}\sup_{\g\in\G_{\min(n,m)}}\|r_n(\g)-r_m(\g)\|<\infty$; 
\item $\sup_{n\in \mathbb N}\|C_n(\g)\|<\infty$, for every $\g \in \G$; 
\item for every $C>0$ there exists $n_C\in\mathbb N$ such that for all $n\geq n_C$  we have $\{\g\in \G_{n+1}\,:\, \|r_{n+1}(\g)\|\leq C\}\subset \G_n$.
\end{enumerate}
 
For every $\g\in\G$ we define a map $r:\G\ra \ell^2(\G)$ by letting $r(\g)=r_n(\g)$, where $n$ is chosen to be the smallest natural number such that $\g\in \G_n$. The above properties then imply that $r\in \mathcal{RQ}(\G, \cup_n\mathcal G_n,\ell^2(\G))$.

The reader may verify that this above construction together with Proposition 1.10 in \cite{CS} shows that if there exists  a sequence $r_n\in\mathcal{RA}(\G_n, \{e\},\ell^2(\G_n))$ with uniform bounded equivariance then $r\in\mathcal{RA}(\oplus_n\G_n, \{e\},\ell^2(\oplus \G_n))$. In particular we have that if $\mathcal{RA}(\G, \{e\},\ell^2(\G))\neq\emptyset$ then $\mathcal{RA}(\G^{\oplus\infty}, \{e\},\ell^2(\G^{\oplus\infty}))\neq\emptyset$.

As expected, to obtain relative quasi-cocycles we have to impose stronger assumptions. For example, if there exist relative quasi-cocycles $r_n \in \mathcal{RQ}(\G_n, \mathcal G_n,\ell^2(\G_n))$  satisfying $\sup_{m,n}\sup_{\g\in\G_{\min(n,m)}}\|r_n(\g)-r_m(\g)\|<\infty$, $\sup_{n\in \mathbb N}D_n<\infty$, and condition (3), the same construction as before shows that $\mathcal{RQ}(\G, \cup_n\mathcal G_n,\ell^2(\G))\neq \emptyset$. We also notice that by a basic rescaling procedure the same conclusion follows if we completely drop the uniform boundedness on the defects $D_n$, keep condition (3), and replace the first condition by the following: there exists a sequence $K_n\geq D_n$ such that  $$\sup_{m,n}\sup_{\g\in\G_{\min(n,m)}}\|\frac{1}{K_n}r_n(\g)-\frac{1}{K_m}r_m(\g)\|<\infty.$$
\end{examples}

The examples presented above arise more or less from canonical algebraic constructions. More interestingly, relative quasi-cocycles on groups can be constructed naturally from purely geometric considerations. Below we single out a class of such examples  which are intensely studied in geometric group theory.

\vskip 0.05in

\noindent {\bf G. Relatively hyperbolic groups.}  The results in \cite{Mon,Min,MMS} imply that every Gromov hyperbolic group $\G$ admits a proper quasi-cocycle into a multiple of $\ell^2(\G)$ (Lemma 4.2 in \cite{Tho}). Using similar reasoning we will show a relative version of this result for groups which are relatively hyperbolic in the sense of Bowditch \cite{Bow}. 

Briefly, given a group $\G$ together with a family of subgroups $\mathcal G$, we say that $\G$ is hyperbolic relative to $\mathcal G$  if there exists a graph $\mathcal K$ on which $\G$ acts such that the following conditions are satisfied: a) $\G$ and every $\Sigma\in \mathcal G$ are finitely generated, b) $\mathcal K$ is fine (see (1) in Definition 2 from \cite{Bow}) and has thin triangles, c) there are finitely many orbits and each edge stabilizer is finite, and d) the infinite vertex stabilizers are precisely the elements of $\mathcal G$ and their conjugates.       

Here are some examples of relatively hyperbolic groups: a free product is relatively hyperbolic with respect to its factors; if $\G$ is hyperbolic relative to a family of subgroups $\mathcal G$ and $\alpha:\Sigma_1\ra \Sigma_2$ is a monomorphism with $\Sigma_i\in\mathcal G$, then the HNN extension $\G\star_\alpha$ is hyperbolic with respect to $\mathcal G\setminus \{\Sigma_1,\Sigma_2\}$ \cite{Dah}; geometrically finite Kleinian groups are hyperbolic with respect to their cusp subgroups \cite{Farb}; the fundamental group of a complete hyperbolic manifold of finite volume is hyperbolic relative to its cusp subgroups \cite{Farb}; Sela's limit groups are hyperbolic relative to their maximal noncyclic abelian subgroups \cite{Dah}.

Mineyev and Yaman \cite{MinYa} showed that whenever $\G$ is hyperbolic relative to a finite set $\mathcal G$ of subgroups, there exists an ideal hyperbolic tuple $(\G,\mathcal G, X, {\nu}')$ (Definition 42 in \cite{MinYa}). Furthermore, using this in combination with the machinery developed in \cite{Min}, they constructed a homological $\mathbb Q$-bicombing in $X$ which is $\G$-equivariant, anti-symmetric, quasi-geodesic, and has bounded area (Theorem 47 in \cite{MinYa}). Therefore, applying the same arguments as in the proof of Theorem 7.13 of \cite{MScocycle}, we see that this bicombing gives rise naturally to relative quasi-cocycles for $\G$ into a multiple of the left semi-regular representations with respect to some conjugates of elements in $\mathcal G$. In effect, the bounded area together with anti-symmetry will imply the quasi-cocycle relation and being quasi-geodesic  will imply properness with respect to the family $\mathcal G$. 

\begin{prop}If a group $\G$ is hyperbolic relative to a finite family of subgroups  $\mathcal G$, then we have that $\mathcal{RQ}(\G, \mathcal G,\oplus_{i,j} \ell^2(\G/\g_j\Sigma_i\g^{-1}_j))\neq \emptyset$, for some $\g_j\in\G$ and $\Sigma_i\in \mathcal G$. \end{prop}

Finally, we mention that from the work of Ozawa it is known that for every group $\G$ that is relatively hyperbolic with respect to a family of amenable subgroups we have that $\mathcal {RA}(\G, \{e\}, \ell^2(\G)^\infty )\neq \emptyset$.


\section{Weak Amenability for Groups and von Neumann Algebras}

The notion of weak amenability for groups was introduced by Cowling and Haagerup in \cite{CoHa}. There are several equivalent definitions (\cite {BoPi, CoHa}) and for the reader's convenience we recall the following: 
\begin{defn} A countable discrete group $\G$ is said to be weakly amenable with constant $C$ if there exists a sequence of finitely supported functions $\phi_n: \G\to \bb C$ such that $\phi_n \ra 1$ pointwise and $\limsup_n \|\widehat\phi_n\|_{cb} \leq C$, where $\|\widehat\phi_n\|_{cb}$ denotes the (completely bounded) norm of the Schur multiplier on $\mathfrak B(\ell^2(\G))$ associated with the kernel $\widehat\phi_n: \G\times \G\to \bb C$ given by $\widehat\phi_n(\g,\de) = \phi_n(\g^{-1}\de)$.\end{defn}

The Cowling-Haagerup constant $\La_{cb}(\G)$ is defined to be the infimum of all $C$ for which such a sequence
$(\phi_n)$ exists. If $\G$ is not weakly amenable then we write $\La_{cb}(\G) = \infty$.  

Below we summarize some families of groups known to be weakly amenable, also specifying their Cowling-Haagerup constants: \begin{enumerate} \item all amenable groups ($\La_{cb}(\G)=1$); \item  all lattices in $SO(n,1)$ and $SU(n,1)$ ($\La_{cb}(\G)=1$) or lattices in $Sp(n,1)$ ($\La_{cb}(\G)=2n-1$), \cite{CoHa}; \item Coxeter groups ($\La_{cb}(\G) = 1$) \cite{Janus}; \item more generally, all groups which act properly on finite dimensional $\rm CAT(0)$-cube complexes ($\La_{cb}(\G)=1$), \cite{GuHi, Mizut};  \item all hyperbolic groups  (in this case no explicit constants were computed), \cite{OzHyp}.\item all limit groups in the sense of Sela ($\La_{cb}(\G)=1$); this is an observation due to Ozawa based on a result from \cite{GHW}. \end{enumerate}

Groups which are not weakly amenable include $\mathbb Z^2\rtimes SL_2(\mathbb Z)$, \cite{Haa} (see also, \cite{Dorof}), lattices in higher-rank simple Lie groups, and any non-amenable wreath products of the form $\mathbb Z \wr \Sigma$, \cite{OzCBAP}.   

 The class of weakly amenable groups is closed under taking subgroups, cartesian products, co-amenable extensions, measure equivalence \cite{OzCBAP}, and inductive limits of groups with uniformly bounded Cowling-Haagerup constants. However, it is not known whether weak amenability is closed under taking a free product of two groups except in the case that the Cowling--Haagerup constants of both groups are one \cite{RicardXu}.

 By analogy with the group case discussed above, one can define a similar approximation property for von Neumann algebras.   The precise formulation is the following.
 
 \begin{defn} A von Neumann algebra $M$ is said to have the \emph{weak* completely bounded approximation property}, abbreviated W*CBAP, if there is a sequence of ultraweakly-continuous finite-rank maps ($\phi_n)$ on $M$ such that $\phi_n \ra \id_M$ in the point-ultraweak topology and $\limsup_n \|\phi_n \|_{cb} < \infty$.\end{defn}

In \cite{OPCartanI} Ozawa and Popa discovered that the presence of this finite-dimensional approximation (with constant one) on a group imposes a certain type of ``rigidity'' on its internal structure. More precisely, they showed that if $\La_{cb}(\G)=1$ then for any amenable subgroup $\Omega<\G$ with non-amenable normalizing group  $N_\G(\Omega)$ there exists an $\Omega\rtimes N_\G(\Omega)$ invariant state on $\ell^\infty(\Omega)$, where the semidirect product $\Omega\rtimes N_\G(\Omega)$ acts on $\Omega$  by $(\g,a)\cdot x=\g ax \g^{-1}$. In other words, the natural action of the normalizer $N_\G(\Omega)$ on $\Omega$ is fairly ``small''; for instance, it cannot be of Bernoulli type. Later, Ozawa showed that in fact \emph{all} weakly amenable groups satisfy this property, \cite{OzCBAP}.  In fact, this rigidity even manifests in the von Neumann--algebraic context, as follows:

\begin{thm}[Ozawa and Popa \cite{OPCartanI}, Ozawa \cite{OzCBAP}] Let $M$ be a von Neumann algebra which has W*CBAP and let  $P\subset M$ be a diffuse amenable subalgebra. Then the natural action by conjugation of the normalizer $\mathcal N_M(P)\ca P$ is \emph{weakly compact}, i.e., there exists a net of positive unit vectors $(\eta_n)_{n\in\mathbf N}$ in $L^2(M)\bar\otimes L^2(\bar M)$ such that:
\begin{enumerate}
\renewcommand{\labelenumi}{(\Alph{enumi})}

\item $\nor{\eta_n - (v \otimes \bar v)\eta_n}\to 0$, for all $v \in \mathcal U(P)$;

\item $\nor{[u \otimes \bar u,\eta_n]}\to 0$, for all $u \in \mathcal N_M(P)$;

\item $\ip{(x \otimes 1)\eta_n}{\eta_n} = \tau (x) = \ip{(1 \otimes \bar x)\eta_n}{\eta_n}$, for all $x \in M$.

\end{enumerate}
\end{thm}

In combination with deformation techniques, weak compactness turned out to be an powerful tool for obtaining many important structural results for group--measure space factors \cite{OPCartanI,OPCartanII,HouSh,Sin}.


\section{The Gaussian Construction, Bimodules and Weak Containment}

Given an orthogonal representation $\pi:\G\ra \mathcal O(\mathcal H_{\mathbb R})$ of a countable, discrete group there exists a way of associating to it  a p.m.p.\ action of $\G$ on a measure space such that the induced Koopman representation is unitarily equivalent to the infinite direct sum of the symmetric tensor powers of $\pi$ (see the proof of Lemma 3.5 in \cite{Vae}). This is called the Gaussian construction associated to 
$(\G,\pi,\mathcal H_{\mathbb R})$. We briefly describe this construction here, indicating how it can be extended to measure preserving actions by product groups. 

If $\pi:\G\ra \mathcal O(\mathcal H_{\mathbb R})$ is an orthogonal representation,  the Gausssian construction as described in  \cite{PeSi} or \cite{Sin} provides a probability measure space $(Y_{\pi},\nu)$ and a family $\omega(\xi)_{\xi\in \mathcal H} $ of unitaries in $L^\infty(Y_{\pi})$ such that $L^\infty (Y_\pi)$ is generated as a von Neumann algebra by the $\omega(\xi)$'s  and the following relations hold:
\begin{enumerate} \item $\omega(0)=1$, $\omega(\xi_1+\xi_2)=\omega(\xi_1)\omega (\xi_2) $,  $\omega(\xi)^*=\omega(-\xi)$ for all $\xi,\xi_1,\xi_2\in \mathcal H_{\mathbb R} $ \item $\tau(\omega (\xi))=\exp(-\|\xi\|^2)$ where $\tau$ is the trace on $L^\infty(Y_{\pi})$ given by integration. \end{enumerate}

\noindent The action $\sigma$ of $\G$ on $L^\infty(Y_\pi)$ is given by $\sigma_g(\omega(\xi) )=\omega (\pi_g(\xi))$, for all $\xi\in\mathcal H_{\mathbb R}$.  

Suppose now that $\G_1\times \G_2$ acts in a trace preserving manner on an abelian von Neumann algebra $(A,\tau) $ and denote by $M=A\rtimes(\G_1\times \G_2)$ the corresponding crossed product von Neumann algebra.  For each $i=1,2$ let $\pi_i:\G_i\ra\mathcal O(\mathcal H_i)$ be an orthogonal representation which is weakly contained in the (real) left regular representation of $\G_i$. Let $L^2(Y_{\pi_i})_0=L^2(Y_{\pi_i})\ominus \mathbb C 1$ be the Koopman representation of the Gaussian action corresponding to $\pi_i$ which, by the assumptions, it is also weakly contained in the left regular representation. Consider the Hilbert space $\mathcal K=L^2(A)\bar \otimes L^2(Y_{\pi_1})_0\bar\otimes L^2(Y_{\pi_2})_0 \bar\otimes \ell^2(\G_1\times \G_2)$ with the $M$-bimodular structure defined as
\begin{equation*}(au_g)\cdot (\xi\otimes \xi_1\otimes \xi_2\otimes \delta_k)\cdot (bu_h)=(a\sigma_g(\xi)\sigma_{gk}(b))\otimes (\pi_g(\xi_1\otimes \xi_2))\otimes (\delta_{gkh} ),\end{equation*} 
for every $a,b \in A$, $\xi \in L^2(A)$, $\xi_1\in L^2(Y_{\pi_1})_0$, $\xi_2\in L^2(Y_{\pi_2})_0$, and  $g,k,h\in \G_1\times \G_2$. Here 
$\pi=\pi_1\otimes\pi_2$.

One of the key ingredients needed in the proof of Theorem \ref{controlweakembedding} is that whenever $A$ is amenable the above $M$-bimodule is weakly contained in the coarse $M$-bimodule. 

\begin{lem}[Fell's absorption principle] \label{weakcontainment} As an $M$-bimodule, $\mathcal K$ is isomorphic with a multiple of $\, L^2(\langle M, A\rangle, Tr )$. In particular, when $A$ is amenable, it follows  that $\mathcal K$ is weakly contained in the coarse bimodule, $L^2(M)\bar \otimes L^2(M)$. 
\end{lem}

\begin{proof} First we notice that when $\pi_i$ is weakly contained in $\rho_i$ then the bimodule associated to the pair $(\pi_1,\pi_2)$ is weakly contained in  the bimodule associated with the pair $(\rho_1,\rho_2)$. It is therefore enough to prove the statement in the case when $\pi_i$ is the (real) left regular representation of $\G_i$. 

Throughout the proof, we denote by $\G=\G_1\times \G_2$. Since $\mathcal K$ is canonically identified with $L^2(A)\bar \otimes \ell^2(\G)\bar \otimes \ell^2(\G)$, we will obtain the desired conclusion by showing that the map 
\begin{equation*}L^2(A)\bar \otimes \ell^2(\G) \bar\otimes \ell^2(\G)\,\ni\,  \xi\otimes \delta_g\otimes \delta_h\ra \xi u_ge_Au_{g^{-1}h} \, \in \, L^2(\langle M, A\rangle, Tr )\end{equation*} implements an isomorphism between the two bimodules. 

To this purpose it suffices to show that
\begin{eqnarray}\label{701}  &&\langle(au_s)\cdot(\xi\otimes \de_g\otimes \de_h)\cdot (bu_t), (a'u_{s'})\cdot(\xi'\otimes \de_{g'}\otimes \de_{h'})\cdot (b'u_{t'})\rangle \\ \nonumber &=&\langle au_s(\xi u_ge_Au_{g^{-1}h})bu_t,a'u_{s'}(\xi' u_{g'}e_Au_{g'^{-1}h'})b'u_{t'}\rangle _{Tr} ,\end{eqnarray}
for all $a,a',b,b'\in A$, $\xi,\xi'\in L^2(A)$, and $s,t,g,h,s',t',g',h'\in \G$.

On the one hand, by definitions, the left side in the previous equation is equal to
\begin{eqnarray*}
&&\langle(au_s) \cdot(\xi\otimes \de_g\otimes \de_h)\cdot (bu_t), (a'u_{s'})\cdot(\xi'\otimes \de_{g'}\otimes \de_{h'})\cdot (b'u_{t'})\rangle \\ &=& \langle (a\sigma_s(\xi)\sigma_{sh}(b))\otimes \de_{sg}\otimes \de_{sht}, (a'\sigma_{s'}(\xi')\sigma_{s'h'}(b'))\otimes \de_{s'g'}\otimes \de_{s'h't'} \rangle \\
&=& \de_{sg,s'g'} \de_{sht,s'h't'} \langle (a\sigma_s(\xi)\sigma_{sh}(b)), (a'\sigma_{s'}(\xi')\sigma_{s'h'}(b'))\rangle \\
&=& \de_{sg,s'g'} \de_{sht,s'h't'} \tau( \sigma_{s'h'}(b'^*) \sigma_{s'}(\xi'^*) a'^*a\sigma_s(\xi)  \sigma_{sh}(b)).\end{eqnarray*}

On the other hand, using basic computations and $\tau (\sigma_{s'g'}(x))=\tau(x)$ for all $x\in A$ we see that the right side of (\ref{701}) is equal to  

\begin{eqnarray*} && \langle au_s(\xi u_ge_Au_{g^{-1}h})bu_t, a'u_{s'}(\xi' u_{g'}e_Au_{g'^{-1}h'})b'u_{t'}\rangle_{Tr}\\
&=&Tr(u_{t'^{-1}}b'^*u_{h'^{-1}g'}e_A u_{g'^{-1}}\xi'^* u
_{s'^{-1}}a'^*au_s\xi u_ge_Au_{g^{-1}h}bu_t)\\
&=& Tr(e_A u_{g'^{-1}}\xi'^* u
_{s'^{-1}}a'^*au_s\xi u_ge_Au_{g^{-1}h}bu_{tt'^{-1}}b'^*u_{h'^{-1}g'}e_A)\\ 
&=& \tau(E_A(u_{g'^{-1}}\xi'^* u
_{s'^{-1}}a'^*au_s\xi u_g) E_A(u_{g^{-1}h}bu_{tt'^{-1}}b'^*u_{h'^{-1}g'}))\\
&=& \tau( E_A(\sigma_{g'^{-1}}(\xi'^*) \sigma_{g'^{-1}s'^{-1}}(a'^*a\sigma_s(\xi))u_{g'^{-1}s'^{-1}sg} ) E_A(\sigma_{g^{-1}h}(b)\sigma_{g^{-1}htt'^{-1}}(b'^*)u_{g^{-1}htt'^{-1}h'^{-1}g'}) )\\
&=& \de_{g'^{-1}s'^{-1}sg,e} \de_{g^{-1}htt'^{-1}h'^{-1}g',e}\tau(\sigma_{g'^{-1}}(\xi'^*) \sigma_{g'^{-1}s'^{-1}}(a'^*a\sigma_s(\xi))\sigma_{g^{-1}h}(b)\sigma_{g^{-1}htt'^{-1}}(b'^*)\\ 
&=& \de_{sg,s'g'} \de_{g^{-1}ht, g'^{-1}h't'}\tau(\sigma_{(s'g')^{-1}}( \sigma_{s'}(\xi'^*) a'^*a\sigma_s(\xi)\sigma_{s'g'g^{-1}h}(b)\sigma_{s'g'g^{-1}htt'^{-1}}(b'^*))\\ 
&=& \de_{sg,s'g'} \de_{sht, s'h't'}\tau(\sigma_{s'}(\xi'^*) a'^*a\sigma_s(\xi) \sigma_{sh}(b)\sigma_{s'h'}(b'^*)).\end{eqnarray*}

\noindent This establishes (\ref{701}) and hence the conclusion of the lemma.  \end{proof}


\section{A Path of Automorphisms of the Extended Roe Algebra Associated with the Products of Gaussian Actions}\label{path} Let $\G=\G_1\times\G_2 \curvearrowright^\sigma X$ be a measure preserving action of $\G$ on a measure space $X$. Assume we are  given orthogonal representations $\pi_i: \G_1\ra \mathcal O(\mathcal H_i)$ . As shown in the previous section, to these representations we can associate the Gaussian actions $ \G_i \curvearrowright^{\pi_i} (Y_{\pi_i},\nu_i)$ (in a slight abuse of notation we will denote the Gaussian action by the same letter). Next we consider the product action $ \G \curvearrowright^{\pi_1\otimes \pi_2} (Y_{\pi_1}\times Y_{\pi_2},\nu_1\times \nu_2)$ and the diagonal action of $ \G$ on $(X\times Y_{\pi_1}\times Y_{\pi_2},\mu\times \nu_1\times \nu_2)$.  To this action, following \cite{CS}, we can associate the extended Roe algebra $C^*_u(\G \curvearrowright Z)$ (where $Z=X\times Y_{\pi_1}\times Y_{\pi_2}$).

Additionally, given any pair of quasi-cocycles $q_i: \G_i\ra \mathcal H_i$ for the respective representations $\pi_i$, $i=1,2$,  we can construct a one-parameter family $(\alpha_t)_{t \in \bb R}$ of $\ast$-automorphisms of $C^*_{u}(\G\ca Z)$, by exponentiating the $q_i$'s. This traces back to the construction of a malleable deformation of $L\G$ from a cocycle $b$ as carried out in \S 3 of \cite{Sin}. Moreover, this family will be pointwise continuous with respect to the uniform norm as $t \to 0$ (Theorem \ref{deform-ineq}).

 Given the quasi-cocycles  $q_i: \G_i \to \mathcal H_i$, one can construct, following section \S 1.2 of \cite{Sin}, two one-parameter families of maps $\up^i_t: \G_i \to \mathcal U(L^\infty(Y_{\pi_i},\nu_i))$ defined by the formula $\up^i_t(\g_i)(x)=\exp(\sqrt{-1}t q_i(\g_i)(x))$, where $\g_i\in \G_i$, $x\in Y_{\pi_i}$, respectively. To understand this formula, the reader must think  about $\mathcal H_i$ as being identified with a subspace of $L^2(Y_{\pi_i},\nu_i)$, viewing the elements $q_i(\g_i)$ as functions  on $Y_{\pi_i}$. The same computations as in \cite{PeSi,Sin} show the following:
\begin{prop} Assuming the same notations as above, we have that: 
\begin{enumerate}
\item If the representation $\pi_i$ is weakly-$\ell^2$, $i=1,2$, then the (tensor) product of Koopman representations ${\pi_1\otimes \pi_2}_{|L_0^2(Y_{\pi_1})\otimes L_0^2(Y_{\pi_2})}$ is also weakly-$\ell^2$;
 \item $\int_Y^{\pi_i} \up^i_t(\g_i)(y) \up^i_t(\de_i)^*(y) d\mu^{\pi_i}(y) = \kappa^i_t(\g_i,\de_i)$, $i=1,2$, and $\g_i,\de_i\in \G_i$.
\end{enumerate}

Here, $\kappa^i_t(\g_i,\de_i) = \exp(-t\|q_i(\g_i) - q_i(\de_i)\|)$.
\end{prop}

With the help of these maps we can construct a path of unitary operators $V_t\in\mathfrak B(L^2(Z)\bar\otimes \ell^2(\G))= \mathfrak B(L^2(Y_{\pi_1})\bar\otimes L^2(Y_{\pi_2})\bar\otimes L^2(X)\bar\otimes \ell^2(\G))$ by letting $V_t(\xi_1\otimes \xi_2\otimes \eta\otimes
\delta_{(\g_1,\g_2)})=\up^1_t(\g_1)\xi_1\otimes \up^2_t(\g_2)\xi_2\otimes \eta \otimes \delta_{(\g_1,\g_2)}$ for every  $\eta\in L^2(X)$,  $\xi_i\in L^2(Y_{\pi_i})$, and $\g_i \in \G_i$, where $i=1,2$. The computations in \cite{CS} show that the $V_t$ enjoy the following basic properties.

\begin{prop}\label{45} For every $t,s\in\mathbb R$ we have that:
\begin{enumerate}
\item $V_tV_s=V_{t+s}$, $V_tV^*_t=V^*_tV_t=1$
\item If the array is anti-symmetric we have $JV_tJ=V_{t}$ and if it is symmetric we have $JV_tJ=V_{-t}$. Here we denoted by  $J: L^2(Z)\bar\otimes \ell^2(\G)\ra L^2(Z)\bar\otimes \ell^2(\G)$ is Tomita's conjugation.
\end{enumerate}\end{prop}

The unitary $V_t$ implements an inner $\star$-automorphism $\alpha_t$ on $\mathfrak B (L^2(Z)\bar\otimes
\ell^2(\G))$ by letting $\alpha_t(x)= V_t xV^*_t$ for all $x\in\mathfrak B (L^2(Z)\bar\otimes \ell^2(\G))$.  The $\alpha_t$ then restricts to  a family of inner automorphisms of the uniform Roe algebra.  Moreover, when restricting to the uniform Roe algebra one can recover from $\alpha_t$ the multipliers introduced in section 2 of \cite{CS} by the formula $\mathfrak m_t ([x_{\g,\de}])=([\kappa_t (\g,\de) x_{\g,\de} ]$). Precisely, we have $E_M\circ\alpha_t(x)=1_X\otimes \mathfrak m_t(x)$ for all $x\in C^*_u(\G)$. The same computations as in \cite{CS} can be used to show that, $\alpha_t$, when restricted to the Roe algebra, is a $C^*$-deformation, i.e., it is pointwise-$\|\cdot \|_\infty$ continuous.

\begin{thm}[Lemma 2.6 in \cite{CS}]\label{deform-ineq} Let $q$ be any symmetric or anti-symmetric array. Assuming the notations above, for every $x\in L^\infty(X)\rtimes_{\sigma,r}\G$ we have
\begin{eqnarray}
\label{3}&&\|(\alpha_t(x)-x)\circ e\|_{\infty}\ra 0\text{ as }t\ra 0;\\
\label{10000}&&\|(\alpha_t(JxJ)-JxJ)\circ e\|_{\infty}\ra 0\text{ as }t\ra 0,
\end{eqnarray}
where $\|\cdot\|_{\infty}$ denotes the operatorial norm in $\mathfrak B(L^2(X)\bar\otimes \ell^2(\G))$; here $e$ denotes the
orthogonal projection from $L^2(Z) \bar\otimes \ell^2(\G)$ onto $L^2(X) \bar\otimes \ell^2(\G)$.\end{thm}

\noindent  The proof of the following result is straightforward and we leave it to the reader.  
\begin{prop}\label{deform-eq-prod} Let $q$ be any symmetric or anti-symmetric array. Assuming the notations above, for every $x\in L^\infty(X)\rtimes_{\sigma,r}\G$, $v\in \mathcal U(L^\infty(X)\rtimes_\sigma\G)$, and every $t\in \mathbb R$  we have
\begin{eqnarray}
&&\|(\alpha_t\otimes 1 (x\otimes v)-x\otimes v)\circ (e\otimes 1)\|_\infty=  \|(\alpha_{t}(x)-x)\circ e\|_{\infty};\\
&&\|(\alpha_t\otimes 1 (\tilde J(x\otimes v)\tilde J)-\tilde J (x\otimes v) \tilde J)\circ (e\otimes 1)\|_{\infty}=  \|(\alpha_{t}(JxJ)-JxJ)\circ e\|_{\infty},
\end{eqnarray}
where on the left hand side of the above formulas  $\|\cdot\|_{\infty}$ denotes the operatorial norm in  $\mathfrak B([L^2(X) \bar\otimes \ell^2(\G)]\bar\otimes [L^2(X) \bar\otimes \ell^2(\G)])$, 
$\tilde J$ is Tomita's conjugation operator on $[L^2(Z) \bar\otimes \ell^2(\G)]\otimes [L^2(X) \bar\otimes \ell^2(\G)]$, and $e\otimes 1$ is the
orthogonal projection from $[L^2(Z) \bar\otimes \ell^2(\G)]\bar\otimes [L^2(X) \bar\otimes \ell^2(\G)]$ onto $[L^2(X) \bar\otimes \ell^2(\G)]\bar\otimes [L^2(X) \bar\otimes \ell^2(\G)]$. 
\end{prop}

\noindent Combining the two previous results we obtain the following 
\begin{cor}\label{deform-ineq-prod} Let $q$ be any symmetric or anti-symmetric array. Assuming the notations above, for every $x\in L^\infty(X)\rtimes_{\sigma,r}\G$ and $v\in \mathcal U(L^\infty(X)\rtimes_\sigma\G)$  we have
\begin{eqnarray*}\label{deform-ineq-prod2}
&&\sup_{\eta \in([L^2(X) \bar\otimes \ell^2(\G)]\bar\otimes [L^2(X) \bar\otimes \ell^2(\G)])_1}\|V_t\otimes 1 (\eta (x\otimes v))-(V_t\otimes 1 (\eta))(x\otimes v)\|\ra 0\text{ and }\\
&&\sup_{\eta \in([L^2(X) \bar\otimes \ell^2(\G)]\bar\otimes [L^2(X) \bar\otimes \ell^2(\G)])_1}\|V_t\otimes 1 ((x\otimes v)\eta)-(x\otimes v) V_t\otimes 1 (\eta)\|\ra 0,
\end{eqnarray*}
as $t\ra 0$.
 \end{cor}


\section{Proofs of the Main Results}

We start by proving the main technical result of the paper which involves product of groups.  Specifically,  we obtain a result  describing all weakly compact embeddings in the crossed product von Neumann algebras arising from actions of products of hyperbolic groups (Theorem \ref{controlweakembedding}). Our approach follows the general outline of the proof of Theorem B in \cite{OPCartanII} and Theorem B in \cite{CS}. However, it is based on a new ingredient which allows us to treat the more general case of arrays rather than just quasi-cocycles as proved in \cite{CS}. This was influenced by the approach taken in \cite{CIfree}.

\begin{thm}\label{intertwiningrelarrqc} Let $\G$ be an exact group together with a family of subgroups $\mathcal G = \lbrace \Sigma_i \rbrace$, and let $\pi:\G\ra \mathcal U(\mathcal H_\pi)$ be a weakly-$\ell^2$ representation.  Also, let $\G\curvearrowright X$ be a free, ergodic action and denote by $M=L^\infty(X)\rtimes \G$ the corresponding crossed-product von Neumann algebra. 

\emph{1.} If $\, \mathcal {RA}(\G,\mathcal G, \mathcal H_\pi)\neq\emptyset$  and $P\subseteq M$ is a diffuse subalgebra, then there exist projections $p_0, p_i \in \mathcal{Z}(P' \cap M)$ such that $p_0 + \bigvee_i p_i = 1$, $(P' \cap M)p_0$ is amenable and $Pp_i \preceq_M L^{\infty}(X) \rtimes \Sigma_i$, for all $i$. In particular, either $P' \cap M$ is amenable or there exists a group $\Sigma \in\mathcal G$ such that $P\preceq_M L^\infty(X)\rtimes \Sigma$.

\emph{2.}  If $\, \mathcal {RQ}(\G,\mathcal G, \mathcal H_\pi)\neq\emptyset$ and $P\subseteq M$ is a weakly compact embedding with $P$ diffuse, then there exist projections $p_0, p_i \in \mathcal{Z}(\mathcal{N}_M(P)' \cap M)$ such that $p_0 + \bigvee_i p_i = 1$, $\mathcal{N}_M(P)''p_0$ is amenable and $Pp_i \preceq_M L^{\infty}(X) \rtimes \Sigma_i$, for all $i$. In particular, either $\mathcal{N}_M(P)''$ is amenable or there exists a group $\Sigma \in\mathcal G$ such that $P\preceq_M L^\infty(X)\rtimes \Sigma$.

\emph{3.}  Assume that $\,\mathcal G$ is a family of normal subgroups of $\G$. If $\, \mathcal {RA}(\G,\mathcal G, \mathcal H_\pi)\neq\emptyset$ and $P\subseteq M$ is a weakly compact embedding with $P$ diffuse, then  there exist projections $p_0, p_i \in \mathcal{Z}(\mathcal{N}_M(P)' \cap M)$ such that $p_0 + \bigvee_i p_i = 1$, $\mathcal{N}_M(P)''p_0$ is amenable and $Pp_i \preceq_M L^{\infty}(X) \rtimes \Sigma_i$, for all $i$. In particular, either $\mathcal{N}_M(P)''$ is amenable or there exists a group $\Sigma \in\mathcal G$ such that $P\preceq_M L^\infty(X)\rtimes \Sigma$.
\end{thm}

\begin{proof} As stated, the first part is Theorem 3.2 in \cite{CS} while the second part follows exactly as in the proof of Theorem 4.1 in \cite{CS}. Indeed the only ingredient needed for this is to adapt Proposition 2.6 in \cite{CS} to the case of quasi-cocycles that are proper with respect to a family of subgroups. This is a straightforward exercise, and we leave it to the reader. So we only prove the third part. 

 We establish the following notations: Let $\{\Sigma_i \,:\,i\in I\}$ be an enumeration of $\mathcal G$, $M_i=L^\infty(X)\rtimes \Sigma_i$ for each $i\in I$, $N=\mathcal N_M(P)''$, and $\mathcal Z= \mathcal Z(N'\cap M)$. Let $p_0\in \mathcal Z$ be the maximal projection such that $Np_0$ is amenable. For each $i$, let $p_i \in (P'\cap M)(1-p_0)$ be a maximal projection satisfying $Pp_i \preceq_M L^{\infty}(X) \rtimes \Sigma_i$ (obtained via a  standard maximality argument). By maximality, we must have that $p_i \in \mathcal Z(P'\cap M)$. Moreover, we have that $p_i\in\mathcal Z$. Indeed, if $u\in\mathcal N_{M}(P)$, let $\tilde p_i=up_iu^*(1-p_0-p_i)$. Then $p_i+\tilde p_i$ satisfies the same condition and by the maximality of $p_i$, we get that $\tilde p_i=0$. Thus $p_i=up_iu^*$, for every $u\in\mathcal N_M(P)$, hence $p_i \in\mathcal Z$.
Therefore, to prove the theorem, we only need to show that $p_0 + \bigvee p_i =1$. By contradiction, assume that $p:=1-(p_0 + \bigvee_i p_i) \not = 0$. Note that $Pp \npreceq_{M} M_i$, for any $i$, by maximality. Also due to the maximality, $Np$ has no amenable direct summand.
 
  By assumption $P\subset M$ is weakly compact, so there exists a net of positive unit vectors $(\eta_n)_{n\in\mathbf N}$ in $L^2(M) \otimes L^2(\bar M)$ such that:
\begin{enumerate}
\renewcommand{\labelenumi}{(\Alph{enumi})}

\item$\nor{\eta_n - (v \otimes \bar v)\eta_n}\to 0$, for all $v \in \mathcal  U(P)$;

\item $\nor{[u \otimes \bar u,\eta_n]}\to 0$, for all $u \in \mathcal  N_M(P)$; and

\item $\ip{(x \otimes 1)\eta_n}{\eta_n} = \tau (x) = \ip{(1 \otimes \bar x)\eta_n}{\eta_n}$, for all $x \in M$.

\end{enumerate}

\noindent So let $\mathcal H = L^2_0(Y^\pi)\otimes L^2(X)\otimes \ell^2(\G)$ which as we remarked before is weakly contained as an $M$-bimodule in the coarse bimodule. Fixing $t>0$ we consider the unitary $V_t$ associated with an array $q$ as defined in the previous sections. Next denote by $\tilde \eta_{n,t}=(V_t\otimes 1)(p\otimes 1)\eta_n$, $\zeta_{n,t} = (e
\otimes1)\tilde\eta_{n,t}= (e\circ  V_t \otimes 1)(p\otimes 1)\eta_n$, and $\xi_{n,t} = \tilde\eta_{n,t} - \zeta_{n,t}=(e^\perp\otimes 1)\tilde\eta_{n,t}\in \mathcal H\otimes L^2(M)$.

 Using these notations we show next the following inequality:

\begin{lem}\label{bigchunk} \begin{equation*}\Lim_n\nor{\xi_{n,t}} \geq \frac{1}{2}\|p\|_2,
\end{equation*}
\noindent where $\Lim$ is an ultralimit on $\mathbb N$.
\end{lem}

\begin{proof}We argue by contradiction, so by passing to a subsequence we can assume that
\begin{equation}\label{1022}\| \xi_{n,t}\|<  \frac{1}{2}\|p\|_2 \text{ for all } n.\end{equation}

\noindent Denoting by $\zeta_n=(p\otimes 1)\eta_{n}$ we have $\|\tilde \eta_{n,t}\|=\|\zeta_n\|=\|p\|_2$ and using the identity $\| (e\otimes 1)(\tilde\eta_{n,t})\|^2+\|(e^\perp\otimes 1)(\tilde\eta_{n,t})\|^2= \|\tilde\eta_{n,t}\|^2=\|p\|_2^2$ in combination with (\ref{1022}) we have \begin{equation}\label{1023}\|(e\otimes 1)(\tilde\eta_{n,t})\|>  \frac{\sqrt 3}{2}\|p\|_2,\text{ for all }n.\end{equation}

The main strategy is to prove that relation (\ref{1023}) together with the assumption $Pp \npreceq_M M_i$, for all $i$ will enable us to show that, after passing to an infinite subsequence of $\zeta_{n}$, one can construct an infinite sequence $\mathcal F_1,\mathcal F_2,\mathcal F_3,\ldots $ of mutually disjoint finite subsets of $\G$ satisfying the following property: there exists $i\in I$ and $1>D>0$ such that for every $k\in \Bbb N$ we can find $n_k\in \Bbb N$ such that for all $j \leq k$ we have \begin{equation}\label{1027}\|P_{\Sigma_i\mathcal F_{j}}(\zeta_{n_k})\|\geq\frac{D}{\sqrt j}\|p\|_2.\end{equation} Here and throughout the proof $P_{\Sigma_i\mathcal F_j}$ stands for the orthogonal projection from $L^2(M) \bar{\otimes} L^2(\bar M)$ onto the closed linear span of the set $ \lbrace  M_iu_g \otimes L^2(\bar M) \,:\, g \in \mathcal F_j \rbrace$.

First, we briefly explain how this claim leads to a contradiction, thus finishing the proof of the lemma. Since the sets $\mathcal F_j$ are disjoint, relation (\ref{1027}) implies $\displaystyle\|p\|^2_2=\|\zeta_{n_k}\|^2\geq  \sum ^k_{j=1}\| P_{\Sigma_i\mathcal F_j}(\zeta_{n_k})\|^2_2\geq D^2\|p\|_2^2 (\sum^k_{j=1} \frac{1}{j}) ,$ for all $k\in \mathbb N$. This is obviously impossible when letting $k$ be sufficiently large.  
\vskip0.05in

\noindent So we are left to prove (\ref{1027}). To show this we will proceed by induction on $k$. 

We begin by establishing the base case $k=1$. Since $\zeta_n\in L^2(M)\bar{\otimes}L^2(\bar M)$, we write
$\zeta_n=\sum_{g\in\G}\zeta^{n}_{g} \de_g $, where $\zeta^n_{g}\in L^2(X)\bar\otimes L^2(\bar{M})$. Then, using the definition of
$V_t$, a straightforward computation shows that \begin{equation*}\|e\otimes 1(\tilde \eta_{n,t})\|^2=\sum_{g\in \G}\exp(-2t^2\|q(g)\|^2)\|\zeta_n^g\|^2,\text{ for all }n.\end{equation*} When combined with (\ref{1023}) this formula implies that, for all $n$ we have  \begin{equation}\label{200}\sum_{g\in\G}\exp(-2t^2\|q(g)\|^2)\|\zeta_n^g\|^2> \frac{3}{4}\|p\|^2_2.\end{equation}

\noindent Since the map $g\ra \|q(g)\|$ is proper relative to the family $\mathcal G$, then the set $\{g\in\G\,  :\,  \exp(-t^2\|q(g)\|^2) \geq \frac{1}{2}\}$ is contained in $\mathcal F' R \mathcal F'$ for some finite subsets $\mathcal F'\subset \G$ and $R\subset \mathcal G$ and, using the inequality (\ref{200}), we further deduce that
$$\frac{3}{4}\|p\|_2^2<\frac{1}{4}\sum_{g\in \G\setminus \mathcal F' R'\mathcal F'}\|\zeta_n^g\|^2+\sum_{g\in \mathcal F'R'\mathcal F'}\|\zeta_n^g\|^2,\text { for all }n,$$ 
\noindent where we have denoted by $R'=\cup_{\Sigma \in R} \Sigma\subset \G$.
 
\noindent  By basic algebraic manipulations, the above inequality gives that $\displaystyle\sum_{g\in\mathcal F'R'\mathcal F'}\|\zeta_n^g\|^2> \frac{2}{3}\|p\|^2_2$ for all $n$ which implies \begin{equation}\label{1024'}\|P_{\mathcal F' R'\mathcal F'} (\zeta_n)\|>\sqrt \frac{2}{3}\|p\|_2.\end{equation}
\noindent Also note that since $\mathcal G$ is a family of normal subgroups then  there exists finite subset $F\subset \G$ such that  $\mathcal F' R'\mathcal F'= R'\mathcal F$ and hence by (\ref{1024'}) we have 
\begin{equation}\label{1024''}\|P_{R'\mathcal F} (\zeta_n)\|>\sqrt \frac{2}{3}\|p\|_2.\end{equation}

Notice that $R'\mathcal F =\cup_{\Sigma \in R} \Sigma \mathcal F$ and if we denote by $s=|R|$ then using the inclusion-exclusion principle together with the triangle inequality in (\ref{1024''}) we see that after passing to an infinite subsequence $\zeta_n$ there exists $\Sigma_i \in R$ such that for all $n$  

\begin{equation}\label{1024}\|P_{\Sigma_i\mathcal F} (\zeta_{n})\|> \frac{1}{\sqrt{2^{2s-1}3}}\|p\|_2.\end{equation}

\noindent So if we set $D= \frac{1}{\sqrt{2^{2s-1}3}}$ then  case $k=1$ follows form (\ref{1024}) by  letting $\mathcal F_1=\mathcal F$ and $n_1=1$. 
\vskip 0.05in
To conclude we now show the inductive step, i.e., assuming that we have constructed the sets $\mathcal F_1,\mathcal F_2,\ldots,\mathcal F_k\subset \G$, we indicate how to construct $\mathcal F_{k+1}\subset \G$ and $n_{k+1}\in \mathbb N$ satisfying (\ref{1027}). Consider the finite set ${\mathcal S} =\cup_{j=1}^k(\mathcal F_j \mathcal F^{-1})\subset \G$ and notice that since $\Sigma_i$ is quasi-normal then the set $\Sigma_iS\Sigma_i$ is small over $\mathcal G$.  Thus, since $\mathcal S$ is finite and $Pp\npreceq_{M}M_i$ for any $i$, by Popa's intertwining techniques, there exist a unitary  $v\in\mathcal U(P)$, a finite set $\mathcal K\subset\G$, and an element $v'$ in the linear span of $\{  u_h\, :\,  h\in \mathcal K\}$ such that \begin{eqnarray}&&\label{51111}\mathcal K \Sigma_i\cap \Sigma_i \mathcal S=\emptyset; \\ \label{5120}&&\|v'-vp\|_2\leqslant \frac{D^2}{6k(k+1)| \mathcal F |}\|p\|_2 .\end{eqnarray}

\noindent Next, we show that for $n\in\Bbb N$ and $z\in M$ we have
\begin{equation}\label{115}\|(z\otimes 1)P_{\Sigma_i\mathcal F}(\zeta_n)\|\leq |\mathcal F|\|z\|_2.\end{equation}

\noindent Fix $n$ and denote by $P$ the orthogonal projection onto $L^2(M_i)\bar{\otimes}L^2(\bar{M})$, i.e., $P=P_{\Sigma_i}$. We have $P_{\Sigma_i\mathcal F}(\zeta_n)=\sum_{h\in{\mathcal F}}P(\zeta_n(u^*_h\otimes 1))(u_h\otimes 1)$ and by the Cauchy-Schwarz inequality we deduce

\begin{equation}\label{114}\|(z\otimes 1)P_{\Sigma_i\mathcal F}(\zeta_n)\|^2
\leqslant |\mathcal F|\sum_{h\in {\mathcal F}}\|(z\otimes 1)P(\zeta_n(u^*_h\otimes 1))\|^2\end{equation}

 Now we let $E_{M_i}$ to be the conditional expectation from $M$ onto $M_i$ and we denote by $a=E_{M_i}(z^*z)^{\frac{1}{2}}$. Using the formulas $\langle (x\otimes 1)P(\zeta),P(\zeta)\rangle=\langle (E_{M_i}(x)\otimes 1)P(\zeta),P(\zeta)\rangle$ and $\|(x\otimes1)\eta_n\|=\|x\|_2$, for all $\zeta\in L^2(M)\bar{\otimes}L^2(\bar{M})$ and $x\in M$, we obtain the following:

\begin{equation*}\begin{split} &\|(z\otimes 1)P(\zeta_n(u^*_h\otimes 1))\|^2 \\ 
&=\langle (z^*z\otimes 1)P(\zeta_n(u^*_h\otimes 1)),P(\zeta_n(u^*_h\otimes 1))\rangle \\
&=\langle (E_{M_i}(z^*z)\otimes 1)P(\zeta_n(u^*_h\otimes 1)),P(\zeta_n(u^*_h\otimes 1))\rangle \\
&=\|(a\otimes 1)P(\zeta_n(u^*_h\otimes 1))\|_2^2= \|P((a\otimes 1)\zeta_n(u^*_h\otimes 1))\|^2\\ 
&\leq \|(a\otimes 1)\zeta_n\|^2=\|ap\|_2^2 \leq \|a\|_2^2=\|z\|_2^2.
\end{split}\end{equation*}

\noindent It is clear that the last inequalities combined with (\ref{114}) give (\ref{115}). Applying the triangle inequality, for all $v\in\mathcal U(P)$ and all $n\in\mathbb N$, we have

\begin{eqnarray}\label{133} \quad&&\|\zeta_n-(v'\otimes \bar{v}) P_{\Sigma_i\mathcal F}(\zeta_n)\|
 \leq \\ \nonumber &&\|\zeta_n-(vp\otimes \bar{v})\zeta_n\|+\|\zeta_n-P_{\Sigma_i\mathcal F}(\zeta_n)\|+\|((v'-vp)\otimes 1)P_{\Sigma_i\mathcal F}(\zeta_n)\|\end{eqnarray}

\noindent Since $p$ and $v$ commute, we have
$\zeta_n-(vp\otimes \bar{v})\zeta_n=(p\otimes 1)(\eta_n-(v\otimes\bar{v})\eta_n)$. Thus, since $\lim_{n\rightarrow\infty}\|\eta_n-(v\otimes\bar{v})\eta_n\|_2=0$, we can find $n_{k+1}\geq n_{k}$ such
that for all $n\geq n_{k+1}$ we have  \begin{equation}\label{201} \|\zeta_n-(vp\otimes \bar{v})\zeta_n\|\leq \frac{D^2}{6k(k+1)}\|p\|_2.\end{equation} Using (\ref{115}) for $z=vp-v'$ in combination
with (\ref{5120})  for all $n$ we have  \begin{equation}\label{202}\|((v'-vp)\otimes 1)P_{\mathcal F}(\zeta_n)\|\leq \frac{D^2}{6k(k+1)}\|p\|_2.\end{equation} Altogether, (\ref{133}), (\ref{201}), (\ref{202}), and (\ref{1024}) show that that for all $n\geq n_{k+1}$ we have

\begin{equation}\label{203}\begin{split}\|\zeta_n-(v'\otimes \bar{v}) P_{\Sigma_i\mathcal F}(\zeta_n)\|&\leq \frac{D^2}{3k(k+1)}\|p\|_2+\|(\zeta_n-P_{\Sigma_i\mathcal F}(\zeta_n))\|\\
&< \left (\frac{D^2}{3k(k+1)} + \left (1-\frac{D^2}{k}\right )^{\frac{1}{2}}\right )\|p\|_2.\end{split}\end{equation}

\noindent Finally, we let $\mathcal F_{k+1}=\mathcal K \mathcal F$ and because $\Sigma_i$ is normal in $\G$ and it follows from (\ref{51111}) that $\Sigma_i \mathcal F_{k+1} $ is disjoint from $\Sigma_i \mathcal F, \Sigma_i \mathcal F_2,\ldots, \Sigma_i\mathcal F_k$. Moreover, we have that $(v'\otimes v)P_{\Sigma_i\mathcal F}(\zeta_n)$ belongs to the closed linear span of $\{ M_{i}u_h \otimes\bar{M} \,  : \, h\in \mathcal F_{k+1}\}$. Thus,
 $P_{\Sigma_i\mathcal F_{k+1}}((v'\otimes \bar{v})P_{\Sigma_i\mathcal F}(\zeta_n))=(v'\otimes \bar{v})P_{\Sigma_i\mathcal F}(\zeta_n)$ and by (\ref{203}) we have that \begin{equation*}\|\zeta_n-P_{\Sigma_i\mathcal F_{k+1}}(\zeta_n)\|<  \left (\frac{D^2}{3k(k+1)} + \left (1-\frac{D^2}{k}\right )^{\frac{1}{2}}\right )\|p\|_2,\text{ for all }n\geq n_{k+1}.\end{equation*} Using this in combination with  $D<1$ then basic calculations show that, for all $ n\geq n_{k+1}$, we have 
 
\begin{equation*}\begin{split}\|P_{\Sigma_i\mathcal F_{k+1}}(\zeta_n)\|&> \left (1- \left (\frac{D^2}{3k^2(k+1)^2} + \left (1-\frac{D^2}{k}\right )^{\frac{1}{2}}\right  )^2\right )^{\frac{1}{2}}\|p\|_2\\
& =\left(\frac{D^2}{k}- \frac{D^4}{3k^4(k+1)^4} -\frac{2D^2}{3k^2(k+1)^2}\left ( 1-\frac{D^2}{k}\right )^{\frac{1}{2}}\right)^{\frac{1}{2}}\|p\|_2\\
&\geq \left(\frac{D^2}{k}- \frac{D^2}{k(k+1)}\right)^{\frac{1}{2}}= \frac{D}{\sqrt{k+1}},\end{split}\end{equation*}

\noindent which ends the proof of (\ref{1027}). \end{proof}

\begin{lem}[Lemma 4.4 in \cite{CS}]\label{almostcomm1}For every $\varepsilon>0$ and every finite set $K\subset L^\infty(X)\rtimes_{\s,r} \G$ with $dist_{\|\cdot\|_2} (y,(N)_1)\leq \e$ for all $y\in K$ one can find $t_\e >0$ and a finite set $L_{K,\e}\subset  \mathcal N_M(P)$ (in this set we allow the situation when the same element is repeated finitely many times) such that
\begin{equation}\label{101}|\langle((yx-xy)\otimes 1)\xi_{n,t},\xi_{n,t}\rangle | \leq  10\e+\sum_{v\in L_{K,\e}}\|[v\otimes \bar v,\eta_n]\| ,\end{equation} for all  $y\in K$, $\|x\|_\infty\leq 1$, $t_\e>t>0$, and $n$.\end{lem}

\begin{note}After publication we discovered that the proof of Lemma 4.4 as given in \cite{CS} contains a minor gap which does not affect the substance of the argument. We take the opportunity to provide a corrected proof.
\end{note}

\begin{proof} Fix $\varepsilon>0$ and $y\in K$. Since $N=\mathcal N_M(P)''$ by the Kaplansky density theorem  there exists a finite set $S_y=\{v_i\}\subset \mathcal N_M(P)$ and scalars $\mu_i$
such that $\|\sum_i \mu_i v_i\|_{\infty}\leq 1$ and \begin{equation}\label{102}\|y-\sum_i \mu_i v_i\|_2\leq 0.5\varepsilon.\end{equation}

\noindent Also by the Kaplansky density theorem there exists contractions $w_i \in L^\infty(X) \rtimes_{\sigma,r} \G$ such that
for all $i$ and all $y\in K$ we have 
\begin{equation}\label{102'}\|w_i^*-v^*_i\|_2=\|w_i-v_i\|_2\leq\frac{\varepsilon}{\sum_{i}|\mu_i|}=:\varepsilon'. \end{equation}

\noindent Since the elements $w_i,y \in L^\infty(X)_{\sigma,r}\rtimes \G$, then using Corollary \ref{deform-ineq-prod} one can find a positive number $t_\varepsilon>0$ such that,  for all $t_\varepsilon>t\geq0$ and all $i$  we have the following seven inequalities:
\begin{equation}\label{103}\begin{split}
&\sup_{n}\nor{V_t\otimes 1 ((p\otimes 1)\eta_n (y^*\otimes 1)) - (V_t\otimes 1((p\otimes 1) \eta_n))(y^*\otimes 1) } \leq \varepsilon;\\
&\sup_{n}\nor{V_t\otimes 1 ((y^*\otimes 1)(p\otimes 1)\eta_n ) - (y^*\otimes 1)(V_t\otimes 1((p\otimes 1) \eta_n)) } \leq \varepsilon;\\
&\sup_{n}\nor{V_t\otimes 1 ((w_i^*\otimes \bar v_i^*)(p\otimes 1)\eta_n ) - (w_i^*\otimes \bar v_i^*)(V_t\otimes 1((p\otimes 1) \eta_n)) } \leq \varepsilon';\\
&\sup_{n}\nor{V_t\otimes 1 ((p\otimes 1)\eta_n (w_i^*\otimes \bar v_i^*)) - (V_t\otimes 1((p\otimes 1) \eta_n))(w_i^*\otimes \bar v_i^*) } \leq 0.5\varepsilon';\\
&\sup_{n}\nor{V_t\otimes 1 ((p\otimes 1)\eta_n (w_i\otimes \bar v_i)) - (V_t\otimes 1((p\otimes 1) \eta_n))(w_i\otimes \bar v_i) } \leq 0.5\varepsilon';\\
&\sup_{n}\nor{V_t\otimes 1 ((w_i\otimes \bar v_i)(p\otimes 1)\eta_n ) - (w_i\otimes \bar v_i)(V_t\otimes 1((p\otimes 1) \eta_n)) } \leq \varepsilon';\\
&\sup_{n}\nor{V_t\otimes 1 ((\sum_i\mu_iw_i-y)\otimes 1)(p\otimes 1)\eta_n ) - (\sum_i\mu_iw_i-y)\otimes 1)(V_t\otimes 1((p\otimes 1) \eta_n)) } \leq \varepsilon.
\end{split}
\end{equation}

 \noindent Next we will proceed in several steps to show inequality (\ref{101}). First we fix $t_\varepsilon>t>0$. Then, using the triangle inequality in combination with $\|x\|_{\infty}\leq 1$, the second inequality in (\ref{103}), and the $M$-bimodularity of $e^{\perp}=1-e$, we see that
\begin{equation*}\label{OP4.12}\begin{split}
& |\langle(x\otimes 1)\xi_{n,t},(y^*\otimes 1)\xi_{n,t}\rangle-\langle (x y\otimes
1)\xi_{n,t},\xi_{n,t}\rangle|\\
&\leq \sup_{n}\nor{V_t\otimes 1 ((y^*\otimes 1)(p\otimes 1)\eta_n ) - (y^*\otimes 1)(V_t\otimes 1((p\otimes 1) \eta_n)) }+\\
&\quad +|\langle(x\otimes1)\xi_{n,t},(e^{\perp}V_ty^*p\otimes 1)\eta_{n}\rangle-\langle (xy\otimes 1)\xi_{n,t},\xi_{n,t}\rangle|\\&\leq \varepsilon+ |\langle(x\otimes 1)\xi_{n,t},(e^{\perp}V_{t}y^*p\otimes 1)\eta_{n}\rangle-\langle(xy\otimes 1)\xi_{n,t},\xi_{n,t}\rangle|\end{split}\end{equation*}

 \noindent Further, the Cauchy-Schwarz inequality together with statement (C) from the definition of weak compactness and (\ref{102}) enable us to see that the last quantity above is smaller than

\begin{equation*}\begin{split}
&\leq \varepsilon+\|((y^*p-\sum_i \bar\mu_i v^*_i)p\otimes 1)\eta_n\|+|\sum_i \mu_i\langle(x\otimes
1)\xi_{n,t},(e^{\perp}V_{t} v^*_ip\otimes 1)\eta_{n}\rangle-\langle (xy\otimes 1)\xi_{n,t},\xi_{n,t}\rangle|\\
&\leq1.5 \varepsilon+|\sum_i  \mu_i\langle(x\otimes\bar v^*_i)\xi_{n,t},(e^{\perp}V_{t} pv^*_i\otimes \bar v^*_i)\eta_{n}\rangle-\langle (xy\otimes1)\xi_{n,t},\xi_{n,t}\rangle|\end{split}\end{equation*}

\noindent Since $v_i$ is a unitary then by applying the triangle inequality several times the last quantity above is smaller than

\begin{equation}\label{1003}\begin{split} 
&\leq 1.5\varepsilon+\sum_i|\mu_i|\|[v^*_i\otimes \bar v^*_i,\eta_n]\|+|\sum_i
 \mu_i\langle(x \otimes \bar v^*_i)\xi_{n,t},(e^{\perp}V_{t} p\otimes 1)(\eta_nv^*_i\otimes \bar v^*_i)\rangle-\langle (xy\otimes 1)\xi_{n,t},\xi_{n,t}\rangle|\\
 &\leq 1.5\varepsilon+\sum_i|\mu_i|\left (\|[v^*_i\otimes \bar v^*_i,\eta_n]\|+ \|(p\otimes 1)\eta_n((w_i^*-v_i^*)\otimes \bar v^*_i)\|\right)+\\
 &\quad+|\sum_i
 \mu_i\langle(x \otimes \bar v^*_i)\xi_{n,t},(e^{\perp}V_{t} p\otimes 1)(\eta_nw_i^*\otimes \bar v^*_i)\rangle-\langle (xy\otimes 1)\xi_{n,t},\xi_{n,t}\rangle|\\
&\leq 1.5\varepsilon+\sum_i|\mu_i|\left (\|[v^*_i\otimes \bar v^*_i,\eta_n]\|+ \|\eta_n((w_i^*-v_i^*)\otimes 1)\|\right)+\\
 &\quad+|\sum_i
 \mu_i\langle(x \otimes \bar v^*_i)\xi_{n,t},(e^{\perp}V_{t} p\otimes 1)(\eta_nw_i^*\otimes \bar v^*_i)\rangle-\langle (xy\otimes 1)\xi_{n,t},\xi_{n,t}\rangle| 
 \end{split}
 \end{equation}
 
\noindent We notice that, since $\eta_n$ is a positive vector and $J$ is an isometry then for all $z\in M$ we have $\|\eta_n(z\otimes 1)\|=\|J(z^*\otimes 1)J\eta_n\|=\|(z^*\otimes 1)\eta_n\|=\|z^*\|_2=\|z\|_2$. Using this identity (right traciality) in combination with (\ref{102'}), the triangle inequality, and the fourth respectively the fifth inequality in (\ref{103}) imply that the last quantity in (\ref{1003}) is smaller than

 \begin{equation*}
 \begin{split}&\leq 2.5\varepsilon+\sum_i|\mu_i|(\|[v^*_i\otimes \bar v^*_i,\eta_n]\|+\sup_{n}\nor{V_t\otimes 1 ((p\otimes 1)\eta_n (w_i^*\otimes \bar v_i^*))-\\
 &\quad  - (V_t\otimes 1((p\otimes 1) \eta_n))(w_i^*\otimes \bar v_i^*) })+\\
 &\quad +|\sum_i
 \mu_i\langle(x \otimes \bar v^*_i)\xi_{n,t},(e^{\perp}V_{t} p\otimes 1(\eta_n))(w_i^*\otimes \bar v^*_i)\rangle-\langle (xy\otimes 1)\xi_{n,t},\xi_{n,t}\rangle|\\
 &\leq 3.5\varepsilon+\sum_i|\mu_i|\|[v^*_i\otimes \bar v^*_i,\eta_n]\|+|\sum_i
 \mu_i\langle(x \otimes \bar v^*_i)\xi_{n,t}(w_i\otimes \bar v_i),(e^{\perp}V_{t}p\otimes 1(\eta_n))\rangle-\langle (xy\otimes 1)\xi_{n,t},\xi_{n,t}\rangle|\\
&= 3.5\varepsilon+\sum_i|\mu_i|\|[v^*_i\otimes \bar v^*_i,\eta_n]\|+\\
&\quad +|\sum_i
 \mu_i\langle(x \otimes \bar v^*_i)(e^{\perp}V_{t} p\otimes 1(\eta_n))(w_i\otimes \bar v_i),\xi_{n,t}\rangle-\langle (xy\otimes 1)\xi_{n,t},\xi_{n,t}\rangle|\\
&\leq 3.5\varepsilon+\sum_i|\mu_i|(\|[v^*_i\otimes \bar v^*_i,\eta_n]\|+\sup_{n}\|V_t\otimes 1 ((p\otimes 1)\eta_n (w_i\otimes \bar v_i)) - (V_t\otimes 1(p\otimes 1 \eta_n))(w_i\otimes \bar v_i) \|)\\
&\quad+|\sum_i
 \mu_i\langle(x \otimes \bar v^*_i)(e^{\perp}V_{t} \otimes 1((p\otimes 1)\eta_n(w_i\otimes \bar v_i)),\xi_{n,t}\rangle-\langle (xy\otimes 1)\xi_{n,t},\xi_{n,t}\rangle|\end{split}\end{equation*}

\noindent Using the right traciality of $\eta_n$, (\ref{102'}), the fifth inequality in (\ref{103}),  and $v_i$ being a unitary which commutes with $p$  we see that the last quantity above is smaller than

\begin{equation*}\begin{split}&\leq 4.5\varepsilon+\sum_i|\mu_i|\left(\|[v^*_i\otimes \bar v^*_i,\eta_n]\|+\|(p\otimes 1)\eta_n((w_i-v_i)\otimes \bar v_i)\|\right )\\
&\quad+|\sum_i
 \mu_i\langle(x \otimes \bar v^*_i)(e^{\perp}V_{t} \otimes 1((p\otimes 1)\eta_n(v_i\otimes \bar v_i)),\xi_{n,t}\rangle-\langle (xy\otimes 1)\xi_{n,t},\xi_{n,t}\rangle|\\
&\leq 5.5\varepsilon+\sum_i|\mu_i|\left(\|[v^*_i\otimes \bar v^*_i,\eta_n]\|+\|[v_i\otimes \bar v_i,\eta_n]\|\right)\\
&\quad+|\sum_i
 \mu_i\langle(x \otimes \bar v^*_i)(e^{\perp}V_{t} \otimes 1((v_i\otimes \bar v_i)(p\otimes 1)\eta_n),\xi_{n,t}\rangle-\langle (xy\otimes 1)\xi_{n,t},\xi_{n,t}\rangle|\\
\end{split}\end{equation*}
 Thus, by the triangle inequality in combination with (\ref{102'}), the (left) traciality of $\eta_n$, and the sixth inequality in (\ref{103}) if we continue above we obtain that
 
\begin{equation*}\begin{split} 
&\leq 5.5\varepsilon+\sum_i|\mu_i|\left(\|[v^*_i\otimes \bar v^*_i,\eta_n]\|+\|[v_i\otimes \bar v_i,\eta_n]\|+\|(((w_i-v_i)p)  v_i)\otimes \bar v_i)\eta_n\|\right)+\\
&\quad +|\sum_i
 \mu_i\langle(x \otimes \bar v^*_i)(e^{\perp}V_{t} \otimes 1((w_i\otimes \bar v_i)(p\otimes 1)\eta_n),\xi_{n,t}\rangle-\langle (xy\otimes 1)\xi_{n,t},\xi_{n,t}\rangle|\\
&\leq 6.5\varepsilon+\sum_i|\mu_i|(\|[v^*_i\otimes \bar v^*_i,\eta_n]\|+\|[v_i\otimes \bar v_i,\eta_n]\|+\\ 
&\quad +\sup_{n}\|V_t\otimes 1 ((w_i\otimes \bar v_i)(p\otimes 1)\eta_n ) - (w_i\otimes \bar v_i)(V_t\otimes 1((p\otimes 1) \eta_n))\|)+\\
&\quad+|
 \langle(x(\sum_i\mu_iw_i) \otimes 1)\xi_{n,t},\xi_{n,t}\rangle-\langle (xy\otimes 1)\xi_{n,t},\xi_{n,t}\rangle|\\
 &\leq 7.5\varepsilon+\sum_i|\mu_i|\left (\|[v^*_i\otimes \bar v^*_i,\eta_n]\|+\|[v_i\otimes \bar v_i,\eta_n]\|\right)+|\langle(x(\sum_i\mu_iw_i-y) \otimes 1)\xi_{n,t},\xi_{n,t}\rangle|.
\end{split}\end{equation*}

\noindent Inequality (\ref{102}) together with (C), the Cauchy-Schwarz inequality, $\|x\|_\infty\leq 1$, and the seventh inequality in (\ref{103}) show that the last quantity above is smaller than

\begin{equation*}\begin{split} &\leq 7.5\varepsilon+\sum_i|\mu_i|\left (\|[v^*_i\otimes \bar v^*_i,\eta_n]\|+\|[v_i\otimes \bar v_i,\eta_n]\|\right)+\|((\sum_i\mu_iw_i-y) \otimes 1)\xi_{n,t}\|\\
&\leq 7.5\varepsilon+\sum_i|\mu_i|\left (\|[v^*_i\otimes \bar v^*_i,\eta_n]\|+\|[v_i\otimes \bar v_i,\eta_n]\|\right)+\|(\sum_i\mu_iw_i-y) \otimes 1)(V_t\otimes 1((p\otimes 1)\eta_n))\|\\
&\leq 8.5\varepsilon+\sum_i|\mu_i|\left (\|[v^*_i\otimes \bar v^*_i,\eta_n]\|+\|[v_i\otimes \bar v_i,\eta_n]\|\right)+\|((\sum_i\mu_iw_i-y)p) \otimes 1)\eta_n\|\\
&=8.5\varepsilon+\sum_i|\mu_i|\left (\|[v^*_i\otimes \bar v^*_i,\eta_n]\|+\|[v_i\otimes \bar v_i,\eta_n]\|\right)+\|(\sum_i\mu_iw_i-y)p\|_2\\
&\leq 10\varepsilon+\sum_i|\mu_i|\left (\|[v^*_i\otimes \bar v^*_i,\eta_n]\|+\|[v_i\otimes \bar v_i,\eta_n]\|\right).
\end{split}\end{equation*}

In conclusion, (\ref{101}) follows from the previous inequalities  if we let $L_{K,\varepsilon}$ be the collection of all elements $v_i, v_i^*$ for all the $i$'s and all $y\in K$ where each $v_i$ (resp. $v_i^*$) is repeated $[|\mu_i|]+1$ times.
\end{proof}

\begin{lem}[Lemma 4.5 in \cite{CS}]\label{almostcomm2}For every $\varepsilon>0$ and any finite set $F_0\subset \mathcal U(N)$ there exist a finite set $F_0\subset F\subset M$, a c.c.p. map  $\varphi_{F,\e}: span (F)\ra L^\infty(X)\rtimes_{\s,r} \G $, and $t_\e>0$ such that
\begin{equation}\label{OP4.16}|\psi_{t_\e}(\varphi_{F,\e}(up)^*x\varphi_{F,\e}(up))-\psi_{t_{\e}}(x)|\leq  116\e,\end{equation} for all $u\in F_0$ and $\|x\|_\infty\leq 1$.
\end{lem}

Next we briefly explain how to use the previous lemmas in order to get the proof of the theorem. First notice that, as an $M$-bimodule, $\mathcal H$ is weakly contained in the coarse bimodule. Then following the same argument as in Theorem B of \cite{OPCartanII} we define a state $\psi_t$ on $\mathcal N =\mathfrak B(\mathcal H)\cap\rho(M^{op})'$. Explicitly, if we denote by $\xi_{n,t}= e\otimes 1(\tilde\eta_{n,t})$ we let $\psi_t(x)=\Lim_n\frac{1}{\|\xi_{n,t}\|^2}\langle (x\otimes 1)\xi_{n,t},\xi_{n,t}\rangle$ for every $x \in \mathcal N $. 

To get the proof, from here on, one can proceed exactly as explained in Theorem 4.1 in  \cite{CS}. Namely we use the same Lemmas 4.3  and 4.4 from \cite{CS} and the final part in the proof of Theorem B in \cite{OPCartanII} to conclude that $Np$ is amenable. We leave the details to the reader.

\end{proof}

Finally, we  note that very recently Popa and Vaes extended the third part to arbitrary families of subgroups; even more impressively, when $\G$ is weakly amenable, they showed the result holds for arbitrary  subalgebras $P$, without the weak compact embedding assumption,  \cite{PoVa12}.

\begin{thm}\label{controlweakembedding} Let $\G_1, \G_2$ be exact groups each having a family of quasi-normal subgroups $\mathcal G_1=\lbrace \Sigma_1^{j} \rbrace$ in $\G_1$ and $\mathcal G_2=\lbrace \Sigma_2^{k} \rbrace$ in $\G_2$. and for each $i$ let $\pi_i:\G_i\ra \mathcal U(\mathcal H_{\pi_i})$ be a weakly-$\ell^2$ representation such that $\mathcal{RA}(\G_i,\mathcal G_i,\mathcal H_{\pi_i})\neq\emptyset$. Let $\G_1\times \G_2\curvearrowright X$ be a measure-preserving action on a probability space and denote by $M= L^\infty(X)\rtimes(\G_1\times \G_2)$. If $P\subset M$ is a weakly compact embedding (cf.\ \cite{OPCartanI}), then one can find projections $p_0$, $p_1^{k}$, $p_2^{j} \in \mathcal Z(\mathcal N_M(P)'\cap M )$ with $p_0 + \bigvee p_1^{k} + \bigvee p_2^{j} =1$ such that the following hold:
\begin{enumerate}\item $\mathcal N_M(P)''p_0$ is amenable; \label{1001} \item $Pp_1^{k}\preceq_M L^\infty(X)\rtimes (\G_1\times \Sigma_2^{k})$, for all $k$; \label{1002}\item  $Pp_2^{j}\preceq_M L^\infty(X)\rtimes (\Sigma_1^{j} \times \G_2)$, for all $j$.
\end{enumerate}
\end{thm}

 \begin{proof}  We establish the following notations: $M_1^{k}=L^\infty(X)\rtimes (\G_1\times\Sigma_2^{k})$, $M_2^{j}=L^\infty(X)\rtimes (\Sigma_1^{j}\times \G_2)$,   $N=\mathcal N_M(P)''$ and $\mathcal Z= \mathcal Z(N'\cap M)$. Let $p_0\in \mathcal Z$ be the maximal projection such that $Np_0$ is amenable. For each $k$, let $p_1^k \in (P'\cap M)(1-p_0)$ be a maximal projection satisfying the condition in (2)(obtained via a  standard maximality argument). Similarly, for each $j$, let $p_2^j\in (P'\cap M)(1-p_0-\bigvee p_1^k)$ be a maximal projection satisfying the condition in (3). By maximality, we must have that $p_1^k,p_2^j \in \mathcal Z(P'\cap M)$. Moreover, we have that $p_1^k,p_2^j\in\mathcal Z$. Indeed, if $u\in\mathcal N_{M}(P)$, let $\tilde p_1^k=up_1^ku^*(1-p_0-p_1^k)$. Then $p_1^k+\tilde p_1^k$ also satisfies (2) and by the maximality of $p_1^k$, we get that $\tilde p_1^k=0$. Thus $p_1^k=up_1^ku^*$, for every $u\in\mathcal U(P)$, hence $p_1^k \in\mathcal Z$.

\vskip 0.05in \noindent Therefore, to prove the theorem, we only need to show that $p_0 + \bigvee p_1^k + \bigvee p_2^j =1$. By contradiction, assume that
$p:=1-(p_0 + \bigvee p_1^k + \bigvee p_2^j) \not=0$. Note that $Pp\npreceq_{M}M_1^k$ and $Pp\npreceq_{M}M_2^j$, for any $j$ and $k$, by maximality. Also, note that $Np$ has no amenable direct summand, for the same reason. Now use the remarks in Examples 2.7. B for $\G = \G_1 \times \G_2$. It follows that $\mathcal{RA}(\G, \mathcal G, \mathcal{H}_1 \bar{\otimes} \mathcal{H}_2) \neq \emptyset$, where $\mathcal G = \lbrace \Sigma_1 \times \G_2 \,:\, \Sigma_1 \in \mathcal{G}_1 \rbrace \cup \lbrace \G_1 \times \Sigma_2 \,:\, \Sigma_2 \in \mathcal{G}_2 \rbrace$. So we can apply the third part of Theorem \ref{intertwiningrelarrqc} for the amenable subalgebra $Pp \subset L^{\infty}(X) \rtimes \G$. Thus either $\mathcal N_M(Pp)''$ is amenable, which is impossible (because it contains $Np$ which is non-amenable), or there is a $\La \in \mathcal G$ such that $Pp \preceq_M L^{\infty}(X) \rtimes \La$. Suppose $\La = \Sigma^{j}_1 \times \G_2$, for some $j$. But this means $Pp \preceq_M M^{j}_2$ for some $j$, which is again a contradiction, and this ends the proof.  \end{proof}  

\vskip 0.2in

\begin{proof}[Proof of Corollary \ref{masanormalizer}] Applying the previous theorem for $A=\bb C1$ and $\Sigma_i=e$ there exist $p_0$, $p_1$, $p_2\in \mathcal Z$ with  $p_0+p_1+p_2=1$  such that $p_0\mathcal N_M(P)''$ is amenable, $p_1B\preceq_M L\G_1$, and  $p_2B\preceq_M L\G_2$. Therefore, the conclusion follows if we show that $p_0=1$. Assuming this is not the case one can find $p_1\neq 0$ such that $p_1B\preceq_M L\G_1$. Then Remark 3.8 in \cite{VaesBimodule} implies that  $L\G_1'\cap M\preceq_M p_1B'\cap M$
and since $L\G_1'\cap M=L\G_2$ then we have $L\G_2\preceq_M p_1B'\cap M$. This is, however, a contradiction because $L\G_2$ is a non-amenable factor while $p_1B'\cap M$ is assumed to be an amenable algebra. \end{proof}

\begin{proof}[Proof of Corollary \ref{$W^*$-strongrigidity}]  Assume that $\La \curvearrowright Y$ is a free, ergodic action which is $W^*$-equivalent to $\G_1\times \G_2\curvearrowright X$. This amounts  to the existence of an $\star$-isomorphism $\psi:  L^\infty(Y)\rtimes \La\ra L^\infty(X)\rtimes (\G_1\times \G_2) $. We will denote by  $A=L^{\infty}(X)$, $B=L^\infty (Y)$, $M=A\rtimes (\G_1\times \G_2)$,  $M_1= A\rtimes \G_1$, and $M_2= A\rtimes \G_2$.

Below we will prove that there exists a unitary $x\in \mathcal U( M)$ such that $x\psi(B)x^*=A$. Notice that since $C=\psi(B)$ Cartan in $M$ its normalizing algebra is non-amenable so by Theorem \ref{controlweakembedding} we can assume that $C\preceq_M M_1$. Therefore one can find nonzero projections $p\in C$, $q\in M_1$, a partial isometry $v\in M$, and a $\star$-homomorphism $\phi: Cp\ra qM_1q$ such that for all $x\in Cp$ we have \begin{equation}\phi(x)v=vx.\end{equation} 

Since $C$ is a maximal abelian subalgebra of $M$ then by Lemma 1.5 in \cite{ioana2011} we can assume that $\phi(Cp) \subset qM_1q$ is also a maximal abelian subalgebra. Fixing $u\in \mathcal N _{pMp}(Cp)$ we can easily see that for all $x\in Cp$ we have 
\begin{equation}\label{601}vuv^*\phi(x)=vuxv^*=vuxu^*uv^*=\phi(uxu^*)vuv^*.\end{equation} 

Notice that  $vuv^*vu^*v^*= \phi(uv^*vu^*)vv^*$ is a projection and hence $vuv^*$ is a partial isometry. Also, applying the conditional expectation $E_{qM_1q}$ to equation (\ref{601}), we obtain that for all $x\in Cp$ we have \begin{equation*}E_{qM_1q}(vuv^*)\phi(x)=\phi(uxu^*)E_{qM_1q}( vuv^*).\end{equation*}

Taking the polar decomposition $E_{qM_1q}( vuv^*)=w_u|E_{qM_1q}( vuv^*)|$, the previous equation entails that $|E_{qM_1q}( vuv^*)|\in \phi(Cp)' \cap qM_1q =\phi(Cp)$ and for all $x\in Cp$ we have
\begin{equation*}w_u\phi(x)=\phi(uxu^*)w_u.\end{equation*}
This implies in particular that $w_uw^*_u,w^*_uw_u\in \phi(Cp)' \cap qM_1q =\phi(Cp)$ and therefore  $w_u\in \mathcal{GN}_{qM_1q}(\phi(Cp))$, the normalizing groupoid of $\phi(Cp)$ in $qM_1q$. Altogether, we have shown that \begin{equation*}E_{qM_1q}(vuv^*)\subseteq\mathcal{GN}_{qM_1q}(\phi(Cp))''. \end{equation*}

\noindent  By \cite{Dye63},  we have that  $\mathcal{GN}_{qM_1q}(\phi(Cp))''=\mathcal{N}_{qM_1q}(\phi(Cp))''$ and since the above containment holds for every $u \in \mathcal{N}_{pMp}(Cp)''$ and  $\mathcal{N}_{pMp}(Cp)''=pMp$ we have that \begin{equation*}E_{qM_1q}(vMv^*)\subseteq \mathcal{N}_{qM_1q}(\phi(Cp))'',\end{equation*} and hence $vv^*M_1vv^*\subseteq \mathcal{N}_{qM_1q}(\phi(Cp))''$. This shows in particular that  $\mathcal{N}_{qM_1q}(\phi(Cp))''$ is non-amenable; therefore, by Theorem B in \cite{CS} we have that $\phi(Cp)\preceq_{M_1} A$. By Remark 3.8 in \cite{VaesBimodule} this further implies that $C\preceq_M A$. Finally, by Theorem \ref{intertwining-conj}, one can find a unitary $x \in \mathcal U(M)$ such that $x\phi(B)x^*=xCx^*= A$. 

In particular, our claim shows that the actions $\G_1\times \G_2\curvearrowright X$ and $\La\curvearrowright Y$ are orbit equivalent. Note that, since $\G_1$ and $\G_2$ have property (T) then so is the product $\G_1\times \G_2$, so it follows from Ioana's Cocycle Superrigidity Theorem \cite{IoaCSR} that the actions  $\G_1\times \G_2\curvearrowright X$ and $\La\curvearrowright Y$ are virtually conjugate. \end{proof}

\begin{cor}Let $\G_i$ be weakly amenable groups and let $\pi:\G_i\ra \mathcal U(\mathcal H_\pi)$ be weakly-$\ell^2$ representations such that $\mathcal {RA}(\G,\{e\},\mathcal H_\pi)\neq \emptyset$ (e.g. $\G_i$ are hyperbolic). If  $\G_1\times \G_2\ca X$ and $\La\ca Y$ are any p.m.p.\ actions such that $\La$ admits an infinite amenable normal subgroup $\Sigma <\G$ for which the restriction $\Sigma\ca Y$ is still ergodic then $\G_1\times \G_2\ca X\ncong_{OE}\La\ca Y$.  \end{cor}

\begin{proof}

We will assume that $\G_1\times \G_2\ca X\cong_{OE}\La\ca Y$ and then show that this leads to a contradiction. Thus there exists a $\star$-isomorphism $\psi:  L^\infty(Y)\rtimes \La\ra L^\infty(X)\rtimes (\G_1\times \G_2) $. We will also denote by  $A=L^{\infty}(X)$, $B=L^\infty (Y)$, $P=\psi(L^\infty (Y)\rtimes \Sigma)$, $M=A\rtimes (\G_1\times \G_2)$,  $M_1= A\rtimes \G_1$, $M_2= A\rtimes \G_2$, and notice that $\psi(B)=A$.

Since the Cowling-Haagerup constant is an \emph{ME}-invariant \cite{CoZi} it follows that $\psi(L\Sigma )$ is a weakly compact embedding in $M$.  Since $\Sigma$ is normal in $\Lambda$, then applying Theorem \ref{controlweakembedding}, we can assume that $\psi(L\Sigma)\preceq_M M_1$ and since $\psi(B)=A$ we conclude that $P\preceq_M M_1$. Therefore, one can find nonzero projections $p\in P$, $q\in M_1$, a partial isometry $v\in M$, and a $\star$-homomorphism $\phi: pPp\ra qM_1q$ such that for all $x\in pPp$ we have \begin{equation}\phi(x)v=vx.\end{equation} 

Since $P$ is an irreducible subfactor of $M$, by Proposition \ref{irredtransfer} we can assume that $\phi(pPp) \subset qM_1q$ is also a irreducible subfactor. Fixing $u\in \mathcal N _{pMp}(pPp)$ we can easily see that for all $x\in pPp$ we have 
\begin{equation}\label{901}vuv^*\phi(x)=vuxv^*=vuxu^*uv^*=\phi(uxu^*)vuv^*.\end{equation} 

Notice that  $vuv^*vu^*v^*= \phi(uv^*vu^*)vv^*$ is a projection and hence $vuv^*$ is a partial isometry. Also, applying the conditional expectation $E_{qM_1q}$ to equation (\ref{901}), we obtain that for all $x\in pPp$ we have \begin{equation*}E_{qM_1q}(vuv^*)\phi(x)=\phi(uxu^*)E_{qM_1q}( vuv^*).\end{equation*}

Taking the polar decomposition $E_{qM_1q}( vuv^*)=w_u|E_{qM_1q}( vuv^*)|$, the previous equation entails that $|E_{qM_1q}( vuv^*)|\in \phi(pPp)' \cap qM_1q =\mathbb C q$ and for all $x\in pPp$ we have
\begin{equation*}w_u\phi(x)=\phi(uxu^*)w_u.\end{equation*}
This implies in particular that $w_uw^*_u,w^*_uw_u\in \phi(pPp)' \cap qM_1q =\mathbb C q$ and therefore  $w_u$ is a scalar multiple of a normalizing unitary in $\in \mathcal{N}_{qM_1q}(\phi(pPp))$. Altogether, we have shown that \begin{equation*}E_{qM_1q}(vuv^*)\subseteq\mathcal{N}_{qM_1q}(\phi(pPp))''. \end{equation*}

\noindent Since the above containment holds for every $u \in \mathcal{N}_{pMp}(pPp)$ and  $\mathcal{N}_{pMp}(pPp)''=pMp$ we have that \begin{equation*}E_{qM_1q}(vMv^*)\subseteq \mathcal{N}_{qM_1q}(\phi(pPp))'',\end{equation*} and hence $vv^*M_1vv^*\subseteq \mathcal{N}_{qM_1q}(\phi(pPp))''$. This shows in particular that  $\mathcal{N}_{qM_1q}(\phi(pPp))''$ is non-amenable; therefore, by Theorem B in \cite{CS} we have that $\phi(pPp)\preceq_{M_1} A$. By Remark 3.8 in \cite{VaesBimodule} this would imply that $P\preceq_M A$, which is an obvious contradiction.  \end{proof}

In the remaining part of the section we explain how the second and third part of Theorem 6.1  can be successfully exploited to produce new examples of von Neumann algebras with either unique Cartan subalgebra or no Cartan subalgebras. With this purpose in mind, we introduce the following definition.

\begin{defn}A subgroup $\Sigma<\G$ is called \emph{weakly malnormal} if there exist finitely many elements $\g_1,\g_2,\ldots, \g_n\in\G$ such that \begin{eqnarray*}\left |\bigcap^n_{i=1}\g_i\Sigma\g^{-1}_i\right |<\infty.\end{eqnarray*} \end{defn}

\noindent Therefore, when  the second and the third part in the intertwining theorem above is combined with Corollary 7 from \cite{HouPoVa} we immediately obtain the following uniqueness (absence) of Cartan subalgebra statement.

\begin{cor} Let $\G$ be a weakly amenable group and let $\pi:\G\ra \mathcal U(\mathcal H_\pi)$ be a weakly-$\ell^2$ representation such that one of the following cases holds: 
\begin{enumerate}
\item $\, \mathcal {RQ}(\G,\mathcal G, \mathcal H_\pi)\neq\emptyset$ for a family of weakly malnormal subgroups $\mathcal G$ of $\G$, or 
\item $\, \mathcal {RA}(\G,\{e\}, \mathcal H_\pi)\neq\emptyset$.\end{enumerate} 

\noindent Also let $\G\curvearrowright X$ be a weakly compact, free action. If $\G$ is as in the first case (1) above  we assume in addition that the restrictions $\Sigma\ca X$ are ergodic for all $\Sigma\in \mathcal G$. 

Then $L^\infty(X)\rtimes \G$ has unique Cartan subalgebra. If in addition $\,\G$ is i.c.c.\ and $\mathcal G$ is a family of malnormal groups then $L\G$ has no Cartan subalgebra. 

\end{cor}

\begin{proof}
Let $B$ be a Cartan subalgebra of $M = L^{\infty}(X) \rtimes \G$. Since $\G$ is weakly amenable and the action is weakly compact, it follows that $M = L^{\infty}(X) \rtimes \G$ is weakly amenable, so the inclusion $B \subset M$ is weakly compact, by \cite{OPCartanI}. In the first case we apply (2) in Theorem 6.1 above and see that, since $\mathcal{N}_{M}(B)''=M$ is non-amenable, $B \preceq_{M} L^{\infty}(X) \rtimes \Sigma$, for some $\Sigma \in \mathcal{G}$. Using Corollary 7 from \cite{HouPoVa} we obtain that in fact $B \preceq_{M} L^{\infty}(X)$, from which it follows, by using Appendix 1 in \cite{PoBe}, that $B$ and $L^{\infty}(X)$ are unitarily conjugated. In the second case we apply (3) in the same Theorem 6.1 above to see that again $B \preceq_{M} L^{\infty}(X)$, and the conclusion follows.
\end{proof}

Employing the same strategy as in the proof of Corollary B.2 from \cite{CS} and using the fact that  the class of weakly amenable groups is closed under taking ME-subgroups, we obtain new structural results for measure equivalence of groups. 

\begin{cor}Let $\G$ be a weakly amenable group and let 
$\pi:\G\ra \mathcal U(\mathcal H_\pi)$ be a weakly-$\ell^2$ representation such that  one of the following holds: either $\, \mathcal {RQ}(\G,\mathcal G, \mathcal H_\pi)\neq\emptyset$  for a family of amenable, malnormal subgroups $\mathcal G$, or $\, \mathcal {RQ}(\G,\{e\}, \mathcal H_\pi)\neq\emptyset$. If $\,\La$ is any $ME$-subgroup of $\,\G$ then $L\La$ is strongly solid i.e., given any diffuse amenable subalgebra $\, A\subseteq L\La$ its normalizing algebra $\mathcal N_{L\La}(A)''$ is still amenable. In particular, every amenable subgroup of $\La$ has amenable normalizer. \end{cor}

The following are examples of groups that satisfy the conditions required in the above corollary: any weakly amenable group that is in the class $\mathcal S$ of Ozawa \cite{OzKurosh}--in particular any weakly amenable group $\G$ that is hyperbolic relative to a family of amenable subgroups (e.g. Sela's limit groups which are weakly amenable and hyperbolic with respect to their noncyclic maximal abelian subgroups \cite{Dah}); any weakly amenable HNN extension $\G\star_\alpha$ of a group $\G$, where $\alpha:\Sigma_1\ra\Sigma_2$ is a monomorphism with $\Sigma_i\in\mathcal G$; any infinite free product $\star_{n\in\mathbb N}\G_n$  where  $\G_n$ is hyperbolic relative to a finite family  $\mathcal G_n$ of malnormal groups, $\La_{cb}(\G_n)=1$ and $\mathcal G_n=\{e\}$ for all but finitely many $n$'s -- in this case we choose $\mathcal G=\cup_n\mathcal G_n$.


\section*{Acknowledgements} The authors would like to thank Adrian Ioana, Narutaka Ozawa, and Jesse Peterson for useful discussions. They are also grateful to the anonymous referees for many useful suggestions.

\bibliographystyle{amsplain}

\end{document}